\documentclass[12pt,letter]{amsart}
\usepackage{esint}
\usepackage{amssymb}
\usepackage{MnSymbol}

\usepackage[
hypertexnames=false, colorlinks, 
linkcolor=black,citecolor=black,urlcolor=black, linktocpage=true
]
{hyperref}
\hypersetup{bookmarksdepth=3}
	
	\normalbaselines

\newtheorem{thm}{Theorem}[section]
\newtheorem{lm}[thm]{Lemma}

\newtheorem{prop}[thm]{Proposition}
\newtheorem{conj}[thm]{Conjecture}

\theoremstyle{definition}
\newtheorem{df}[thm]{Definition}
\newtheorem*{df*}{Definition}

\theoremstyle{remark}
\newtheorem{rem}[thm]{Remark}
\newtheorem*{rem*}{Remark}

\numberwithin{equation}{section}

\newcommand{\ci}[1]{_{ {}_{\scriptstyle #1}}}
\newcommand{\ti}[1]{_{\scriptstyle \text{\rm #1}}}

\newcommand{\e}{\varepsilon}

\newcommand{\cD}{\mathcal{D}}

\newcommand{\cS}{\mathcal{S}}

\newcommand{\R}{\mathbb{R}}
\newcommand{\Z}{\mathbb{Z}}

\newcommand{\wt}{\widetilde}

\newcommand{\cz}{Calder\'{o}n--Zygmund\ }

\newcommand{\La}{\langle }
\newcommand{\Ra}{\rangle }

\newcommand{\fdot}{\,\cdot\,}

{\end{list}}

\begin{document}

\title[Sparse square functions and separated bump conditions]{Two-weight estimates for sparse square functions and the separated bump conjecture}
\author{Spyridon Kakaroumpas}
\date{}

\begin{abstract}
We show that two-weight $L^2$ bounds for sparse square functions (uniform with respect to sparseness constants, and in both directions) do not imply a two-weight $L^2$ bound for the Hilbert transform. We present an explicit counterexample, making use of the construction due to Reguera--Thiele from \cite{reguera-thiele}. At the same time, we show that such two-weight bounds for sparse square functions do not imply both separated Orlicz bump conditions on the involved weights for $p=2$ (and for Young functions satisfying an appropriate integrability condition). We rely on the domination of $L\log L$ bumps by Orlicz bumps observed by Treil--Volberg in \cite{entropy} (for Young functions satisfying an appropriate integrability condition).
\end{abstract}

\maketitle
\setcounter{tocdepth}{1}

\tableofcontents
\setcounter{tocdepth}{2}

\section{Introduction and main results}
\label{s: intro}

This paper concerns the relation between two-weight estimates for sparse square functions and the so-called ``separated bump'' conjecture. Here, by weight on $\R^d$ we mean any function on $\R^d$ that is locally integrable, nonnegative, and positive on a set of positive Lebesgue measure. It has been long known that if $1<p<\infty$ and $w$ is a weight on $\R^d$ with $w>0$ a.e., then the celebrated Muckenhoupt $A_p$ condition
\begin{equation}
\label{Muckenhoupt-one-weight}
\sup_{Q}\left(\frac{1}{|Q|}\int_{Q}w(x)dx\right)\left(\frac{1}{|Q|}\int_{Q}w(x)^{-1/(p-1)}dx\right)^{p-1}<\infty,
\end{equation}
where supremum is taken over all cubes $Q$ in $\R^d$ and $|\fdot|,dx$ denote Lebesgue measure on $\R^d$, is sufficient for the boundedness over $L^{p}(w)=L^{p}(w(x)dx)$ of any \cz operator $T$ on $\R^d$, and necessary for that boundedness in the case that $T$ is the Hilbert transform on $\R$, or more generally the vector-valued Riesz transform (that is, all Riesz transforms considered together) on $\R^d$. Note that the boundedness of $T$ over $L^{p}(w)$ is equivalent (with equal norms) to the boundedness of the operator $f\mapsto T(f\sigma)$, denoted in the sequel by $T(\fdot \sigma)$, acting from $L^{p}(\sigma)$ into $L^{p}(w)$, where $\sigma:=w^{-1/(p-1)}$. Note that if \eqref{Muckenhoupt-one-weight} holds, then $\sigma$ is a weight as well.

However, simple examples show that for $1<p<\infty$ and general weights $w,\sigma$ on $\R^d$, the two-weight $A_p$ condition
\begin{equation}
\label{Muckenhoupt-two-weight}
\sup_{Q}\left(\frac{1}{|Q|}\int_{Q}w(x)dx\right)\left(\frac{1}{|Q|}\int_{Q}\sigma(x)dx\right)^{p-1}<\infty,
\end{equation}
where supremum is taken over all cubes $Q$ in $\R^d$, is not sufficient for the boundedness of the operator $T(\fdot \sigma)$ from $L^{p}(\sigma)$ into $L^{p}(w)$, in the case that $T$ is the Hilbert transform on $\R$; condition \eqref{Muckenhoupt-two-weight} is still necessary for that boundedness, though. It should be noted that if $\sigma>0$ a.e., then the boundedness of $T(\fdot \sigma)$ from $L^{p}(\sigma)$ into $L^{p}(w)$ is equivalent (with equal norms) to the boundedness of $T$ from $L^{p}(v)$ into $L^{p}(w)$, where $v:=\sigma^{1-p}$.

It is natural to ask whether ``bumping'' condition \eqref{Muckenhoupt-two-weight} could eliminate its lack of sufficiency for two-weight boundedness. It was proved by C. J. Neugebauer \cite{neugebauer} that if $1<p<\infty$ and $w,v$ are weights with $v>0$ a.e., then for all $r>1$, the condition
\begin{equation}
\label{Neugebauer}
\sup_{Q}\left(\frac{1}{|Q|}\int_{Q}w(x)^rdx\right)\left(\frac{1}{|Q|}\int_{Q}v(x)^{-r/(p-1)}dx\right)^{p-1}<\infty,
\end{equation}
where supremum is taken over all cubes $Q$ in $\R^d$, is sufficient for the boundedness of any \cz operator $T$ from $L^{p}(v)$ into $L^{p}(w)$. In fact, Neugebauer \cite{neugebauer} proved that if condition \eqref{Neugebauer} holds for some $r>1$ then one can find a Muckenhoupt $A_p$ weight $u$ such that $w\leq u\leq c v$ for some positive constant $c$.

More generally, for any Young function $\Phi$, that is a continuous, increasing and convex function $\Phi:[0,\infty)\rightarrow[0,\infty)$ with $\Phi(0)=0$ and $\lim_{t\rightarrow\infty}\Phi(t)/t=\infty$, and for any measure space $(X,\mu)$ denote by $\Vert\fdot\Vert\ci{L^{\Phi}(X,\mu)}$ the Luxemburg norm on $(X,\mu)$ with respect to $\Phi$, given by
\begin{equation*}
\Vert f\Vert\ci{L^{\Phi}(X,\mu)}:=\inf\left\lbrace\lambda>0:~\int_{X}\Phi\left(\frac{|f(x)|}{\lambda}\right)d\mu(x)\leq 1\right\rbrace.
\end{equation*}
The Orlicz space $L^{\Phi}(X,\mu)$ is then defined as the set of measurable functions $f$ on $X$ with $\Vert f\Vert\ci{L^{\Phi}(X,\mu)}<\infty$. We refer to \cite{extrapolation} for more details on Young functions and Orlicz norms. For all cubes $Q$ on $\R^d$, we will denote $L^{\Phi}(Q,dx/|Q|)$, where $dx/|Q|$ is normalized Lebesgue measure on $Q$, by just $L^{\Phi}(Q)$. The following result, having been conjectured (in a slightly different, but equivalent for sufficiently regular Young functions, form) by D. Cruz-Uribe and C. P\'{e}rez in \cite{cruz-uribe-perez}, was proved almost simultaneously by F. Nazarov, A. Reznikov, S. Treil and A. Volberg \cite{nazarov--reznikov--treil--volberg}, and (in a slightly different, but equivalent for sufficiently regular Young functions, form) A. Lerner \cite{lerner} (the latter actually proved a similar result for all $1<p<\infty$, which had also been conjectured by Cruz-Uribe and P\'{e}rez in \cite{cruz-uribe-perez}).

\begin{thm}\label{thm: bump}
Let $\Phi_1,\Phi_2$ be Young functions such that $\int_{c}^{\infty}\frac{1}{\Phi_i(x)}dx<\infty,~i=1,2$, for some $c>0$. Let $w,\sigma$ be weights on $\R^d$ satisfying
\begin{equation}
\label{bump condition}
\sup_{Q}\Vert w\Vert\ci{L^{\Phi_1}(Q)}\Vert\sigma\Vert\ci{L^{\Phi_2}(Q)}<\infty,
\end{equation}
where supremum is taken over all cubes $Q$ in $\R^d$. Then, for any \cz operator $T$, the operator $T(\fdot w)$ is bounded from $L^2(w)$ into $L^2(\sigma)$.
\end{thm}

\subsection{Separated bump conjecture and motivation for main results} It is natural to ask whether it is possible to separate the two bumps appearing in the supremum in \eqref{bump condition} in Theorem \ref{thm: bump}. Cruz-Uribe, Reznikov and Volberg \cite{cruz-uribe--reznikov-volberg} conjectured (in a slightly different, but equivalent for sufficiently regular Young functions, form) the following, which is one version of the ``separated bump'' conjecture (for $p=2$). Denote $\La w\Ra\ci{Q}:=\frac{1}{|Q|}\int_{Q}w(x)dx$.

\begin{conj}\label{conj: separated bump}
Let $\Phi_1,\Phi_2$ be Young functions such that $\int_{c}^{\infty}\frac{1}{\Phi_i(x)}dx<\infty,~i=1,2$, for some $c>0$. Let $w,\sigma$ be weights on $\R^d$ satisfying
\begin{equation}
\label{separated bump condition}
\sup_{Q}\Vert w\Vert\ci{L^{\Phi_1}(Q)}\La\sigma\Ra\ci{Q}<\infty,\qquad\sup_{Q}\La w\Ra\ci{Q}\Vert\sigma\Vert\ci{L^{\Phi_2}(Q)}<\infty,
\end{equation}
where both suprema are taken over all cubes $Q$ in $\R^d$. Then, for any \cz operator $T$, the operator $T(\fdot w)$ is bounded from $L^2(w)$ into $L^2(\sigma)$.
\end{conj}

In \cite{cruz-uribe--reznikov-volberg}, Cruz-Uribe, Reznikov and Volberg establish a special case of Conjecture \ref{conj: separated bump} (these authors establish in \cite{cruz-uribe--reznikov-volberg} an analogous result for all $1<p<\infty$ as well).

\begin{thm}[Cruz-Uribe, Reznikov,Volberg \cite{cruz-uribe--reznikov-volberg}]\label{thm: partial separated bump}
Let $\Phi$ be a Young function such that one of the following holds.

(1) There exists $\delta>0$, such that
\begin{equation*}
\Phi(t)=t(\log(e+t))^{1+\delta},\qquad 0\leq t<\infty.
\end{equation*}

(2) There exists $\delta>1$, such that
\begin{equation*}
\Phi(t)=t\log(e+t)(\log(\log(e^e+t)))^{1+\delta},\qquad 0\leq t<\infty.
\end{equation*}
Let $w,\sigma$ be weights on $\R^d$ with $\sigma>0$ a.e. on $\R^d$, such that
\begin{equation*}
\sup_{Q}\Vert w\Vert\ci{L^{\Phi}(Q)}\La\sigma\Ra\ci{Q}<\infty,\qquad\sup_{Q}\La w\Ra\ci{Q}\Vert\sigma\Vert\ci{L^{\Phi}(Q)}<\infty,
\end{equation*}
where both suprema are taken over all cubes $Q$ in $\R^d$. Then, for any \cz operator $T$, the operator $T(\fdot w)$ is bounded from $L^2(w)$ into $L^2(\sigma)$.
\end{thm}

Numerous other partial results and extensions of them to more general contexts have since been achieved regarding Conjecture \ref{conj: separated bump}, see for instance \cite{log-bumps}, \cite{lacey}, \cite{rahm-spencer}, \cite{entropy}. The proofs of all these results rely on reducing the problem to establishing two-weight bounds for sparse operators, via domination of \cz operators by sparse operators. The latter technique, having been introduced by Lerner \cite{lerner}, has become standard for proving weighted estimates in recent years. For a sparse family $\cS$ of cubes in $\R^d$ (see Subsection \ref{s: sparse definitions} for the relevant definitions), we define the sparse (Lerner) operator corresponding to $\cS$ by
\begin{equation}
\label{sparse-Lerner}
\mathcal{A}\ci{\cS}f:=\sum_{Q\in\cS}\La |f|\Ra\ci{Q}1\ci{Q},~\forall f\in L^1\ti{loc}(\R^d).
\end{equation}
Various results (and in various senses) of domination of \cz operators by sparse operators have been proved, even extending this domination to other classes of operators of interest in harmonic analysis and to the vector-valued setting. We mention just one such result, which motivates the main results of the present paper. Its proof can be found (for instance) in \cite{lacey-A_2} or \cite{lerner-sparse-pointwise} (see also \cite{convex_body} for an extension to the vector-valued setting).

\begin{thm}\label{thm: sparse domination}
For any \cz operator $T$ on $\R^d$, for any compactly supported integrable function $f$ on $\R^d$, there exists an $\eta$-sparse family $\cS$ of cubes in $\R^d$ (in the sense of Definition \ref{df: sparse}) depending on $d,T$ and the function $f$, where $\eta=1-\frac{3^{-d}}{2}$, such that
\begin{equation*}
|Tf|\leq C\mathcal{A}\ci{\cS}f\qquad\text{a.e. on }\R^d,
\end{equation*}
where $C=C(d,T)$ depends on $d,T$ but not on the function $f$.
\end{thm}

One can also consider generalized sparse operators. Namely, for any $1\leq p<\infty$ and for any sparse family $\cS$ of cubes in $\R^d$, define the sparse $p$-function (sparse square function if $p=2$) corresponding to $\cS$ by
\begin{equation}
\label{sparse-p}
\mathcal{A}\ci{\cS,p}f:=\left(\sum_{Q\in\cS}(\La|f|\Ra\ci{Q})^{p}1\ci{Q}\right)^{1/p},~\forall f\in L^1\ti{loc}(\R^d),
\end{equation}
and the sparse maximal function corresponding to $\cS$ by
\begin{equation*}
\mathcal{M}\ci{\cS}f:=\sup_{Q\in\cS}\La|f|\Ra\ci{Q}1\ci{Q},~\forall f\in L^1\ti{loc}(\R^d).
\end{equation*}
It is worth noting that the Hardy--Littlewood maximal function on $\R^d$ admits pointwise sparse domination by sparse maximal functions in the sense of Theorem \ref{thm: sparse domination}. Although the separated bump conjecture has not been proved yet in full generality for sparse (Lerner) operators of the type \eqref{sparse-Lerner}, it turns out that the following strengthened form of the conjecture is true for sparse square functions,  with only mild additional assumptions of regularity for the involved Young functions. It follows immediately by combining Lemma 3.3 and Theorem 4.2 in \cite{entropy}, due to Treil and Volberg.

\begin{prop}[Treil--Volberg \cite{entropy}]\label{p: entropy-sparse-square}
Let $\Phi$ be a Young function with $\int^{\infty}_{c}\frac{1}{\Phi(t)}dt<\infty$ for some $c>0$. Assume in addition that $\Phi$ is doubling and that the function $t\mapsto \Phi(t)/(t\log t)$ is increasing for sufficiently large $t$. Let $w,\sigma$ be weights on $\R^d$ such that
\begin{equation*}
\sup_{Q}\Vert w\Vert\ci{L^{\Phi}(Q)}\La \sigma\Ra\ci{Q}:=K<\infty,
\end{equation*}
where the supremum is taken over all cubes $Q$ in $\R^d$. Then, for any $0<\eta<1$, for any $\eta$-sparse family $\cS$ of cubes in $\R^d$, there holds
\begin{equation*}
\Vert \mathcal{A}\ci{\cS,2}(\fdot w)\Vert\ci{L^2(w)\rightarrow L^2(\sigma)}\leq C(\Phi)c(\eta)K^{1/2}.
\end{equation*}
\end{prop}

In view of the above proposition, one might hope to prove the separated bump conjecture for $p=2$, and for mildly regular Young functions, by establishing that two-weight bounds for sparse square functions, uniform with respect to the sparseness constant of the underlying sparse family, and in both directions, imply two-weight estimates for singular integral operators. More precisely, one might conjecture the following.

\begin{conj}\label{conj: conjecture 1}
Let $w,\sigma$ be weights on $\R^d$ such that for any $0<\eta<1$, for any $\eta$-sparse family $\cS$ of cubes in $\R^d$, there holds
\begin{equation*}
\Vert \mathcal{A}\ci{\cS,2}(\fdot w)\Vert\ci{L^2(w)\rightarrow L^2(\sigma)},~\Vert \mathcal{A}\ci{\cS,2}(\fdot \sigma)\Vert\ci{L^2(\sigma)\rightarrow L^2(w)}\leq C=C(\eta,d).
\end{equation*}
Then, for any \cz operator $T$ on $\R^d$, the operator $T(\fdot w)$ is bounded from $L^2(w)$ into $L^2(\sigma)$.
\end{conj}

One might even be tempted to conjecture the following stronger result, in the spirit of sparse domination of \cz operators.

\begin{conj}\label{conj: conjecture 2}
For any \cz operator $T$ on $\R^d$, for any measurable bounded compactly supported functions $f,g$ on $\R^d$, there exist $\eta=\eta(d,T)\in(0,1)$ and an $\eta$-sparse family $\cS$ of cubes in $\R^d$ (in the sense of Definition \ref{df: sparse}) depending on $d,T$ and the functions $f,g$, such that
\begin{equation*}
|\La T(f),g\Ra|\leq C(\La \mathcal{A}\ci{\cS,2}f,|g|\Ra+\La |f|,\mathcal{A}\ci{\cS,2}g\Ra).
\end{equation*}
where $C=C(d,T)$ depends only on $d,T$.
\end{conj}

\subsection{Main results} The main result of the present paper is that Conjecture \ref{conj: conjecture 1}, and thus also Conjecture \ref{conj: conjecture 2}, is false for the Hilbert transform on $\R$.

\begin{prop}\label{p: main}
For any $1<p<\infty$, there exist weights $w,\sigma$ on $\R$, such that for any $0<\eta<1$, for any $\eta$-sparse family $\cS$ of intervals in $\R$, one has the two-weight bounds
\begin{equation*}
\Vert \mathcal{A}\ci{\cS,p}(\fdot w)\Vert\ci{L^{p}(w)\rightarrow L^{p}(\sigma)},~\Vert \mathcal{A}\ci{\cS,p'}(\fdot \sigma)\Vert\ci{L^{p'}(\sigma)\rightarrow L^{p'}(w)}\leq C=C(\eta,p),
\end{equation*}
but the operator $H(\fdot w)$ is unbounded from $L^{p}(w)$ into $L^{p}(\sigma)$, where $H$ denotes the Hilbert transform.
\end{prop}

We prove Proposition \ref{p: main} by providing an explicit counterexample, using the construction due to M. C. Reguera and C. Thiele \cite{reguera-thiele}, itself a simplified version of the construction due to Reguera \cite{reguera}. The original construction due to Reguera \cite{reguera} was used in \cite{reguera} to disprove the so-called weak-type Muckenhoupt--Wheeden conjecture for the martingale transforms. In \cite{reguera-thiele}, Reguera and Thiele used their construction to disprove the so-called weak-type Muckenhoupt--Wheeden conjecture for the Hilbert transform, while in \cite{reguera-scurry} Reguera and J. Scurry used the construction from \cite{reguera-thiele} to disprove the so-called strong-type Muckenhoupt--Wheeden conjecture for the Hilbert transform. The latter conjecture concerned joint two-weight estimates between the Hardy--Littlewood maximal function and singular integrals, see Subsection \ref{s: previous work}.

\subsubsection{Investigating separated bump conditions} Although the example we provide does disprove Conjecture \ref{conj: conjecture 1}, it fails in an essential way to disprove the separated bump conjecture (Conjecture \ref{conj: separated bump}) itself, for $p=2$.

\begin{prop}\label{p: submain}
One can find weights $w,\sigma$ on $\R$ satisfying Proposition \ref{p: main} for $p=2$, such that for all Young functions $\Phi$ with $\int_{c}^{\infty}\frac{1}{\Phi(t)}dt<\infty$ for some $c>0$,  there holds
\begin{equation*}
\sup_{I}\Vert w\Vert\ci{L^{\Phi}(I)}\La\sigma\Ra\ci{I}=\infty,
\end{equation*}
where supremum is taken over all intervals $I$ in $\R$.
\end{prop}

In particular, two-weight bounds for sparse square functions of the type appearing in Conjecture \ref{conj: conjecture 1} do not imply both separated Orlicz bump conditions on the involved weights for $p=2$, and for Young functions satisfying the integrability condition of Proposition \ref{p: submain}. In order to prove Proposition \ref{p: submain}, we rely on the domination of $L\log L$ bumps by Orlicz bumps (for Young functions satisfying the integrability condition of Proposition \ref{p: submain}) observed by Treil and Volberg in \cite{entropy}.

It is a curious fact that the above issue disappears if one restricts attention to \emph{triadic} intervals.

\begin{prop}\label{p: submain imp}
One can find weights $w,\sigma$ on $\R$ satisfying Proposition \ref{p: main} for $p=2$ as well as Proposition \ref{p: submain}, such that for some $\delta\in(0,1)$, the Young function $\Phi$ given by
\begin{equation*}
\Phi(t):=t\log(e+t)(\log(\log(e^e+t)))^{1+\delta},\qquad0\leq t<\infty,
\end{equation*}
satisfies
\begin{equation*}
\sup_{I\in\mathcal{T}}\Vert w\Vert\ci{L^{\Phi}(I)}\La\sigma\Ra\ci{I}<\infty,\qquad \sup_{I\in\mathcal{T}}\Vert\sigma\Vert\ci{L^{\Phi}(I)}\La w\Ra\ci{I}<\infty,
\end{equation*}
where $\mathcal{T}$ is the family of all triadic intervals in $\R$.
\end{prop}

We emphasize that the exponent $1+\delta$ in Proposition \ref{p: submain imp} satisfies $1+\delta<2$; compare with Theorem \ref{thm: partial separated bump}.

\textbf{Acknowledgements.} I am grateful to Professor Sergei Treil for suggesting the problem of this paper to me, and for several very helpful discussions on various aspects of this paper.

\section{Overview of the paper}

Here we give an outline of the proofs of the main results. In the sequel, the notation $A\lesssim_{c_1,c_2,\ldots}B$, respectively $A\gtrsim_{c_1,c_2,\ldots}B$, will mean that $A\leq CB$, respectively $A\geq CB$, for some positive constant $C$ depending only on the quantities $c_1,c_2,\ldots$. The notation $A\sim_{c_1,c_2,\ldots} B$ will mean that $A\lesssim_{c_1,c_2,\ldots} B$ and $B\lesssim_{c_1,c_2,\ldots} A$.

\subsection{Recalling the estimates from \cite{reguera-scurry} and \cite{reguera-thiele}}\label{s: previous work} In this paper, we make use of a construction of a particular weight on $[0,1)$ due to Reguera--Thiele \cite{reguera-thiele}, which was also used by Reguera--Scurry \cite{reguera-scurry}. For every positive integer parameter $k>3000$, Reguera--Thiele \cite{reguera-thiele} construct a weight $w_k$ on $[0,1)$ for which there exists a ``large'' measurable subset $E_k$ of $\lbrace w_k>0\rbrace$ such that
\begin{equation}
\label{pointwise Hilbert transform estimate}
|Hw_k(x)|\gtrsim kw_{k}(x),\qquad\text{ for all }x\in E_k.
\end{equation}
The lattice of triadic subintervals of $[0,1)$ plays a fundamental role in their construction. For the reader's convenience, we recall the construction of the weight $w_k$ due to Reguera--Thiele \cite{reguera-thiele} in Section \ref{s: reguera-thiele}. Fix $1<p<\infty$. Reguera--Thiele \cite{reguera-thiele} for $p=2$ and Reguera--Scurry \cite{reguera-scurry} for any $1<p<\infty$ define the weight $\sigma_k:=\frac{w_k}{(Mw_k)^{p}}$ on $[0,1)$, where $M$ denotes the Hardy--Littlewood maximal function. Reguera--Scurry \cite{reguera-scurry} prove the estimate
\begin{equation}
\label{pointwise maximal function estimate}
Mw_k(x)\lesssim w_k(x),\qquad\text{ for all }x\in\lbrace w_k>0\rbrace,
\end{equation}
a restricted version of which is also established by Reguera--Thiele \cite{reguera-thiele}. Combining \eqref{pointwise Hilbert transform estimate} and \eqref{pointwise maximal function estimate}, Reguera--Scurry \cite{reguera-scurry} immediately deduce that
\begin{equation}
\label{norm Hilbert transform estimate}
\Vert Hw_k\Vert\ci{L^{p}(\sigma_k)}\gtrsim k\Vert1\ci{[0,1)}\Vert\ci{L^{p}(w_k)}.
\end{equation}
Using \eqref{pointwise maximal function estimate}, Reguera--Scurry \cite{reguera-scurry} establish the two-weights bounds
\begin{equation}
\label{norm maximal function estimate} 
\Vert M(\fdot w_k)\Vert\ci{L^{p}(w_k)\rightarrow L^{p}(\sigma_k)},~\Vert M(\fdot \sigma_k)\Vert\ci{L^{p'}(\sigma_k)\rightarrow L^{p'}(w_k)}\lesssim_{p} 1.
\end{equation}
We emphasize that estimate \eqref{norm maximal function estimate} is uniform with respect to $k$. Using \eqref{norm Hilbert transform estimate} and \eqref{norm maximal function estimate}  and applying a standard ``direct sum of singularities'' (also known as ``gliding hump'') type argument, Reguera--Scurry \cite{reguera-scurry} produce weights $w,\sigma$ on $\R$ such that
\begin{equation*}
\Vert M(\fdot w)\Vert\ci{L^{p}(w)\rightarrow L^{p}(\sigma)},~\Vert M(\fdot \sigma)\Vert\ci{L^{p'}(\sigma)\rightarrow L^{p'}(w)}<\infty,
\end{equation*}
but there exists $f\in L^{p}(w)$ with $H(fw)\notin L^{p}(\sigma)$, thus disproving the strong-type Muckenhoupt--Wheeden conjecture.

\subsection{Two-weight estimates for generalized sparse operators} In order to prove Proposition \ref{p: main}, we begin by establishing two-weight estimates for sparse $p$-functions and sparse $p'$-functions, with respect to the weights constructed by Reguera--Thiele \cite{reguera-thiele} for $p=2$ and Reguera--Scurry \cite{reguera-scurry} for any $1<p<\infty$. Denote $u(Q):=\int_{Q}u(x)dx$.

\begin{prop}\label{p: main estimate}
Let $0<\e<1$, and let $\cS$ be a martingale $\e$-sparse family of subintervals of $[0,1)$. Then, there holds
\begin{equation}
\label{norm sparse square function estimates}
\sum_{\substack{I\in\mathcal{S}\\I\subseteq L}}(\La w_k\Ra\ci{I})^{p}\La\sigma_k\Ra\ci{I}|I|\lesssim_{p} k\frac{w_k(L)}{1-\e},\qquad \sum_{\substack{I\in\mathcal{S}\\I\subseteq L}}(\La \sigma_k\Ra\ci{I})^{p'}\La w_k\Ra\ci{I}|I|\lesssim_{p} k\frac{\sigma_k(L)}{1-\e},
\end{equation}
for any subinterval $L$ of $[0,1)$.
\end{prop}

We refer to Subsection \ref{s: sparse definitions} for the definition of martingale sparse families, and to Subsection \ref{s: proof of local testing} for the proof of Proposition \ref{p: main estimate}. Notice that in contrast to the estimates \eqref{norm maximal function estimate} obtained by Reguera--Scurry \cite{reguera-scurry}, estimates \eqref{norm sparse square function estimates} are not uniform with respect to $k$. To rectify this, we will need to rescale one of the two weights. Namely, pick $r\in(\max(1/(p-1),1),p')$, where $p':=p/(p-1)$, and define the weight $\wt{w}_k:=k^{-r}w_{k}$ on $[0,1)$. Then, it is immediate to see that for any martingale $\e$-sparse family $\cS$ of subintervals of $[0,1)$ there holds
\begin{equation}
\label{rescaled norm sparse square function estimates}
\sum_{\substack{I\in\mathcal{S}\\I\subseteq L}}(\La\wt{w}_k\Ra\ci{I})^{p}\La\sigma_k\Ra\ci{I}|I|\lesssim_{p}\frac{\wt{w}_k(L)}{1-\e},\qquad \sum_{\substack{I\in\mathcal{S}\\I\subseteq L}}(\La \sigma_k\Ra\ci{I})^{p'}\La\wt{w}_k\Ra\ci{I}|I|\lesssim_{p} \frac{\sigma_k(L)}{1-\e},
\end{equation}
for any subinterval $L$ of $[0,1)$. Although the estimates in \eqref{rescaled norm sparse square function estimates} are formally only truncated and restricted versions of the two-weight bounds we would like to prove, and concern only martingale sparse families, they imply the desired two-weight bounds in full generality. Indeed, it is a special case of a result due to A. Culiuc \cite{culiuc} that $L^p$ two-weight bounds for martingale $p$-functions are equivalent to Sawyer-type testing conditions like the one appearing in the first estimate in \eqref{rescaled norm sparse square function estimates}. Moreover, as explained in \cite{convex_body}, estimates for sparse operators with respect to general sparse families can be reduced to estimates for sparse operators with respect to martingale sparse families. We refer to Subsection \ref{s: main sparse estimates} for details.

It is important to note that the rescaling we introduced above does not destroy the blow-up of the norm of the Hilbert transform established by Reguera--Scurry \cite{reguera-scurry}. Indeed, estimate \eqref{norm Hilbert transform estimate} is immediately seen to imply the estimate
\begin{equation}
\label{rescaled norm Hilbert transform estimate}
\Vert H\wt{w}_k\Vert\ci{L^{p}(\sigma_k)}\gtrsim k^{1-(r/p')}\Vert1\ci{[0,1)}\Vert\ci{L^{p}(\wt{w}_k)},
\end{equation}
where $1-(r/p')>0$. Applying a ``direct sum of singularities'' type argument following the one used by Reguera--Scurry \cite{reguera-scurry}, we get weights $\wt{w},\sigma$ on $\R$  such that for any $0<\eta<1$ and for any $\eta$-sparse family $\cS$ of intervals in $\R$ there holds
\begin{equation*}
\Vert\mathcal{A}\ci{\cS,p}(\fdot\wt{w})\Vert\ci{L^{p}(\wt{w})\rightarrow L^{p}(\sigma)},~\Vert\mathcal{A}\ci{\cS,p'}(\fdot\sigma)\Vert\ci{L^{p'}(\sigma)\rightarrow L^{p'}(\wt{w})}\lesssim_{p,\eta}1,
\end{equation*}
but there exists $f\in L^{p}(\wt{w})$ such that $H(f\wt{w})\notin L^{p}(\sigma)$. We refer to Subsection \ref{s: main sparse estimates} for details. This concludes the proof of Proposition \ref{p: main}.

\begin{rem}
Using the original construction due to Reguera \cite{reguera} and applying the same techniques and strategy as in Section \ref{s: sparse-square-estimates}, one can prove an analogue of Proposition \ref{p: main} for the martingale transforms in the place of the Hilbert transform. The details are left to the interested reader.
\end{rem}

\subsection{Investigating separated bump conditions} \label{s: overview bump conditions} Although the example introduced in the previous subsection suffices to disprove Conjecture \ref{conj: conjecture 1}, it fails dramatically to disprove the separated bump conjecture (Conjecture \ref{conj: separated bump}) itself, for $p=2$.

We assume throughout this subsection that $p=2$. We prove that for all Young functions $\Phi$ with $\int_{c}^{\infty}\frac{1}{\Phi(t)}dt<\infty$ for some $c>0$, there holds
\begin{equation*}
\sup\ci{I}\Vert\wt{w}\Vert\ci{L^{\Phi}(I)}\La \sigma\Ra\ci{I}=\infty,
\end{equation*}
where supremum is taken over all intervals $I$ in $\R$. It suffices to prove that one can find subintervals $R_k,~k>4000$ of $[0,1)$ such that for all Young functions $\Phi$ with $\int_{c}^{\infty}\frac{1}{\Phi(t)}dt<\infty$ for some $c>0$, there holds
\begin{equation}
\label{Orlicz blow-up}
\lim_{k\rightarrow\infty}\Vert\wt{w}_k\Vert\ci{L^{\Phi}(R_k)}\La \sigma_k\Ra\ci{R_k}=\infty.
\end{equation}
We need to estimate Orlicz bumps for the weight $w_k$ constructed by Reguera--Thiele \cite{reguera-thiele} from below. Instead of doing this directly, we estimate certain Lorentz bumps, resulting in a dractic simplification of the required computations.

More precisely, consider the function $\phi_{0}:[0,1]\rightarrow[0,\infty)$ given by
\begin{equation*}
\phi_0(s):=s(1-\log s),~0<s\leq1,\qquad\phi_0(0):=0.
\end{equation*}
For all intervals $I$ in $\R$, we denote by $\Lambda\ci{\phi_0}(I)$ the Lorentz space  with fundamental function $\phi_0$, with respect to normalized Lebesgue measure on $I$. We refer to Subsection \ref{s: lorentz norms} for the definition of Lorentz spaces, and to Subsection \ref{s: Llog L} for further remarks on the Lorentz space $\Lambda_{\phi_0}$.

It is an observation due to Treil and Volberg \cite{entropy}, that for all Young functions $\Phi$ with $\int_{c}^{\infty}\frac{1}{\Phi(t)}dt<\infty$ for some $c>0$, one has the estimate
\begin{equation}
\label{entropy-Orlicz comp}
\Vert f\Vert\ci{\Lambda\ci{\phi_0}(I)}\lesssim\ci{\Phi}\Vert f\Vert\ci{L^{\Phi}(I)},
\end{equation}
for all measurable functions $f\geq 0$ on $I$ and for all intervals $I$. We refer to Subsection \ref{s: comparison principles} for details.

In light of \eqref{entropy-Orlicz comp}, it suffices to prove the following stronger result: one can find subintervals $R_k,~k>4000$ of $[0,1)$ such that
\begin{equation}
\label{entropy blow-up}
\lim_{k\rightarrow\infty}\Vert\wt{w}_k\Vert\ci{\Lambda\ci{\phi_0}(R_k)}\La\sigma_{k}\Ra\ci{R_k}=\infty.
\end{equation}
After recalling that $\wt{w}_{k}:=k^{-r}w_k$, where $1<r<2$, this follows from the lemma below, proved in Subsection \ref{s: entropy-Orlicz comparison}.

\begin{lm}\label{l: main entropy estimate}
For all $k>4000$, there exists a subinterval $R_k$ of $[0,1)$ such that
\begin{equation*}
\Vert w_k\Vert\ci{\Lambda\ci{\phi_0}(R_k)}\gtrsim 3^{k}\La w_{k}\Ra\ci{R_k},\qquad\La w_k\Ra\ci{R_k}\La\sigma_k\Ra\ci{R_k}\gtrsim 1.
\end{equation*}
\end{lm}

\subsubsection{An improvement for triadic intervals} The construction due to Reguera--Thiele \cite{reguera-thiele} relies on the triadic structure of the unit interval. It is a curious fact that if one restricts attention to \emph{triadic} intervals, then the situation regarding separated bump conditions improves. 

Namely, consider the Young function $\Phi$ given by
\begin{equation*}
\Phi(t):=t\log(e+t)(\log(\log(e^e+t)))^r,\qquad0\leq t<\infty,
\end{equation*}
where we recall that $1<r<2$. We show that
\begin{equation*}
\sup_{I\in\mathcal{T}}\Vert\wt{w}\Vert\ci{L^{\Phi}(I)}\La \sigma\Ra\ci{I}<\infty,\qquad \sup_{I\in\mathcal{T}}\Vert\sigma\Vert\ci{L^{\Phi}(I)}\La\wt{w}\Ra\ci{I}<\infty,
\end{equation*}
where $\mathcal{T}$ is the family of all triadic intervals in $\R$. The main estimate one has to prove is that
\begin{equation}
\label{triadic Orlicz finite}
\sup_{I\in\mathcal{T}([0,1))}\Vert\wt{w}_{k}\Vert\ci{L^{\Phi}(I)}\La \sigma_{k}\Ra\ci{I}\lesssim_{r}1,\qquad \sup_{I\in\mathcal{T}([0,1))}\Vert\sigma_{k}\Vert\ci{L^{\Phi}(I)}\La\wt{w}_{k}\Ra\ci{I}\lesssim_{r} 1,
\end{equation}
where $\mathcal{T}([0,1))$ is the family of all triadic subintervals $I$ of $[0,1)$. As previously, instead of directly estimating Orlicz bumps, we estimate certain Lorentz bumps.

More precisely, consider the function $\psi:[0,1]\rightarrow[0,\infty)$ given by
\begin{equation*}
\psi(s):=s(12-\log s)(\log(12-\log s))^{r},\qquad 0<s\leq1,\qquad \psi(0):=0.
\end{equation*}
Standard facts on rearrangement-invariant Banach function spaces, Lorentz spaces and Orlicz spaces imply that for all non-atomic probability spaces $(X,\mu)$ there holds
\begin{equation*}
\Vert f\Vert\ci{L^{\Phi}(X,\mu)}\lesssim_{r}\Vert f\Vert\ci{\Lambda_{\psi}(X,\mu)},
\end{equation*}
for all measurable functions $f\geq 0$ on $X$. We refer to Subsection \ref{s: Lorentz-Orlicz comparison} for details.

Therefore, recalling that $\wt{w}_{k}:=k^{-r}w_{k}$, it suffices to prove that
\begin{equation}
\label{uniform triadic lorentz estimate}
\sup_{I\in\mathcal{T}([0,1))}\Vert w_{k}\Vert\ci{\Lambda_{\psi}(I)}\La \sigma_{k}\Ra\ci{I}\lesssim k^r,\qquad \sup_{I\in\mathcal{T}([0,1))}\Vert\sigma_{k}\Vert\ci{\Lambda_{\psi}(I)}\La w_{k}\Ra\ci{I}\lesssim 1.
\end{equation}
The proof of \eqref{uniform triadic lorentz estimate} is given in Subsection \ref{s: Lorentz-Orlicz comparison}.

Unfortunately, in view of Lemma \ref{l: main entropy estimate} it does not seem possible to get an estimate like the first one in \eqref{triadic Orlicz finite} for general subintervals of $[0,1)$ by merely rescaling the weights $w_k,\sigma_k$, and preserve at the same time the blow-up of the norm of the Hilbert transform.

\section{The Reguera--Thiele \cite{reguera-thiele} construction}
\label{s: reguera-thiele}

We recall here the construction due to Reguera and Thiele \cite{reguera-thiele}, and also used by Reguera and Scurry \cite{reguera-scurry}. For definiteness, by interval we mean a subset of $\R$ of the form $[a,b)$, where $a,b\in\R$, $a<b$.

Let $k$ be a positive integer greater than 3000. For every interval $I$, Reguera--Thiele \cite{reguera-thiele} denote by $I^{m}$ its middle triadic child. Then, they define inductively collections $\mathbf{K}_{i},\mathbf{J}_{i}$ of triadic subintervals of $[0,1)$ as follows:
\begin{equation*}
\mathbf{K}_{0}:=\lbrace [0,1)\rbrace,
\end{equation*}
\begin{equation*}
\mathbf{J}_{i}:=\lbrace K^{m}:~K\in\mathbf{K}_{i-1}\rbrace,~i=1,2,\ldots,
\end{equation*}
\begin{equation*}
\mathbf{K}_{i}:=\bigcup_{J\in\mathbf{J}_{i}}\lbrace\text{triadic subintervals of }J\text{ of length } 3^{1-k}|J|\rbrace,~i=1,2,\ldots.
\end{equation*}

Reguera--Thiele \cite{reguera-thiele} set $\mathbf{J}:=\bigcup_{i=1}^{\infty}\mathbf{J}_{i}$, and for all $J\in\mathbf{J}$ they choose a triadic interval $I(J)$ adjacent to $J$ and of length $3^{1-k}|J|$. Reguera--Thiele \cite{reguera-thiele} choose whether to place each $I(J)$ to the right or to the left of $J$ via an inductive scheme, in a way that allows them to establish the estimate 
\begin{equation*}
|Hw_k(x)|\geq (k/3)w_{k}(x),\qquad\text{ for all }x\in E_k
\end{equation*}
for the Hilbert transform $H$, where $w=w_k$ is the weight on $[0,1)$ constructed in \cite{reguera-thiele} (we recall that construction below) and $E_k$ is the set
\begin{equation*}
E_{k}:=\bigcup_{J\in\mathbf{J}}(I(J))^{m}.
\end{equation*}
We refer to \cite{reguera-thiele} for the relevant details. We note that these choices will not play an essential role in the estimates for sparse $p$-functions below.

\begin{rem}
Induction shows that

(a) $\#\mathbf{K}_{i}=3^{i(k-1)}$ and each interval in $\mathbf{K}_{i}$ has length $3^{-ik}$, for all $i=0,1,2,\ldots$

(b) $\#\mathbf{J}_{i}=3^{(i-1)(k-1)}$ and each interval in $\mathbf{J}_{i}$ has length $3^{-(i-1)k-1}$, for all $i=1,2,\ldots$.
\end{rem}

Next, Reguera--Thiele \cite{reguera-thiele} define a weight $w$ on $[0,1)$, given as the pointwise and weak in $L^1([0,1))$ limit of a sequence $w^{(-1)},w^{(0)},w^{(1)},\ldots$ of weights on $[0,1)$. The latter weights are constructed inductively as follows.

One begins by setting $w^{(-1)}:=1\ci{[0,1)}$. Assuming now that for some $i\geq0$ one has defined $w^{(i-1)}$, one obtains the weight $w^{(i)}$ in the following way. The weight $w^{(i)}$ is defined to coincide with $w^{(i-1)}$ outside $\bigcup_{K\in\mathbf{K}_{i}}K$. Moreover, for all $K\in\mathbf{K}_{i}$ the restriction $w^{(i)}|\ci{K}$ of $w^{(i)}$ on $K$ is defined to be
\begin{equation*}
w^{(i)}|\ci{K}:=w^{(i-1)}(K)\frac{1\ci{K^{m}\cup I(K^{m})}}{|K^{m}\cup I(K^{m})|}.
\end{equation*}
Reguera--Thiele \cite{reguera-thiele} point out that $\lbrace w>0\rbrace=\bigcup_{J\in\mathbf{J}}I(J)$, and that the weight $w$ is constant on each $I(J)$.

Fix $p\in(1,\infty)$. Reguera--Thiele \cite{reguera-thiele} for $p=2$, and  Reguera--Scurry \cite{reguera-scurry} for any $1<p<\infty$, define the weight $\sigma:=\frac{w}{(Mw)^p}$ on $[0,1)$, where $M$ is the Hardy-Littlewood maximal function. Reguera--Scurry \cite{reguera-scurry} prove that for all $J\in\mathbf{J}$ and for all $x\in I(J)$ there holds $Mw(x)\leq 13w(x)$, so that in fact $\sigma\sim_{p} w^{1-p}1\ci{\lbrace w>0\rbrace}$.

Let us in what follows simply define
\begin{equation*}
\sigma:=w^{1-p}1\ci{\lbrace w>0\rbrace}=w^{-1/(p'-1)}1\ci{\lbrace w>0\rbrace},
\end{equation*}
where $p':=\frac{p}{p-1}$ is the H\"{o}lder conjugate exponent to $p$.

We set throughout $\mathbf{K}:=\bigcup_{i=0}^{\infty}\mathbf{K}_{i}$. Moreover, for all $i=0,1,2,\ldots$ and for all $K\in\mathbf{K}_i$ we let $\text{ch}\ci{\mathbf{K}}(K)$ be the family of all intervals in $\mathbf{K}_{i+1}$ that are contained in $K$.

\begin{rem}\label{r: vanish}
For all $K\in\mathbf{K}$, setting $J:=K^{m}$ we notice that $w,\sigma$ vanish in $K\setminus(J\cup I(J))$.
\end{rem}

\subsection{Estimating averages over intervals}

In this subsection we estimate averages over intervals of the weights $w,\sigma$. We begin with triadic intervals.

\begin{lm}\label{average_estimates_triadic}
The following hold.

(a) We have
\begin{equation*}
w|\ci{I(J)}\equiv\left(\frac{3^{k}}{3^{k-1}+1}\right)^{i},~\sigma|\ci{I(J)}\equiv\left(\frac{3^{k}}{3^{k-1}+1}\right)^{-(p-1)i},~\forall J\in\mathbf{J}_{i},~\forall i=1,2,\ldots,
\end{equation*}
and
\begin{equation*}
\La w\Ra\ci{K}=\left(\frac{3^{k}}{3^{k-1}+1}\right)^{i},~\La \sigma\Ra\ci{K}= c_{k,p}3^{-k}\left(\frac{3^{k}}{3^{k-1}+1}\right)^{-(p-1)i},~\forall K\in\mathbf{K}_{i},~\forall i=0,1,2,\ldots,
\end{equation*}
where $1\lesssim_{p}c_{k,p}\lesssim1$.

(b) For all $K\in\mathbf{K}$, for all $K'\in\text{ch}\ci{\mathbf{K}}(K)$, and for all triadic subintervals $L$ of $J:=K^{m}$ that are not contained in any interval in $\text{ch}\ci{\mathbf{K}}(K)$, we have
\begin{equation*}
\La w\Ra\ci{K}\sim\La w\Ra\ci{L}=\La w\Ra\ci{K'}=\La w\Ra\ci{I(J)},\qquad \La \sigma\Ra\ci{K}\sim_{p}\La \sigma\Ra\ci{L}=\La \sigma\Ra\ci{K'}\sim_{p}3^{-k}\La \sigma\Ra\ci{I(J)},
\end{equation*}
\begin{equation*}
w(K)\sim3^{k}w(I(J))=3^{k}w(K'),\qquad\sigma(K)\sim_{p}\sigma(I(J))\sim_{p}3^{k}\sigma(K').
\end{equation*}
\end{lm}

\begin{proof}
(a) For all $i=0,1,2,\ldots$, for all $K\in\mathbf{K}_{i}$, and for all $K'\in\mathbf{K}_{i+1}$ with $K'\subseteq K$, we have $w(K)=w^{(i-1)}(K)=w^{(i)}(K)$ and
\begin{equation*}
w(I(J))=w(K')=\frac{w(K)}{3^{k-1}+1}.
\end{equation*}
The claimed formulas for $w$ follow then by induction. Moreover, we have
\begin{align*}
\sigma(K)&=\sum_{l=i+1}^{\infty}\sum_{\substack{J\in\mathbf{J}_{l}\\J\subseteq K}}\sigma(I(J))=
\sum_{l=i+1}^{\infty}3^{(l-i-1)(k-1)}\fdot3^{-lk}\left(\frac{3^{k}}{3^{k-1}+1}\right)^{-(p-1)l}\\
&=3^{-k(i+1)}\left(\frac{3^{k}}{3^{k-1}+1}\right)^{-(p-1)i}\sum_{l=1}^{\infty}3^{-l+1}\left(\frac{3^{k}}{3^{k-1}+1}\right)^{-(p-1)l},
\end{align*}
therefore
\begin{align*}
\La \sigma\Ra\ci{K}=3^{-k}\left(\frac{3^{k}}{3^{k-1}+1}\right)^{-(p-1)i}\frac{3a_{k,p}}{1-a_{k,p}},
\end{align*}
where
\begin{equation*}
a_{k,p}:=\frac{(3^{k-1}+1)^{p-1}}{3^{k(p-1)+1}}\in\left(\frac{1}{3^{p}},\frac{1}{3}\right).
\end{equation*}

(b) Immediate from (a), after noting that $L$ can be written as the disjoint union of triadic subintervals of $J$ of length $3^{1-k}|J|$, all of which are by definition elements of $\text{ch}\ci{\mathbf{K}}(K)$, and that $|K'|=3^{1-k}|J|=|I(J)|=3^{-k}|K|$.
\end{proof}

We now turn to general intervals.

\begin{lm}\label{average_estimates}
Let $I$ be a subinterval of $[0,1)$. Let $K$ be the smallest interval in $\mathbf{K}$ containing $I$. Set $J:=K^{m}$.

(a) If $I\cap J=\emptyset$, then
\begin{equation*}
\La w\Ra\ci{I}\leq\La w\Ra\ci{I(J)},~\La \sigma\Ra\ci{I}\leq\La\sigma\Ra\ci{I(J)}.
\end{equation*}

(b) If $I$ intersects $J$, but not $I(J)$, then
\begin{equation*}
\La w\Ra\ci{I}\lesssim\La w\Ra\ci{K},~\La \sigma\Ra\ci{I}\lesssim_{p}\La \sigma\Ra\ci{K}.
\end{equation*}

(c) If $I$ intersects both $J$ and $I(J)$, then
\begin{equation*}
\La w\Ra\ci{I}\lesssim\La w\Ra\ci{I(J)},~\La\sigma\Ra\ci{I}\lesssim_{p}\max\left(\frac{|I\cap I(J)|}{|I|},3^{-k}\right)\La\sigma\Ra\ci{I(J)}\leq\frac{|I(J)|}{|I|}\La\sigma\Ra\ci{I(J)}.
\end{equation*}
\end{lm}
\begin{proof}
(a) Clear, since $w(I)=w(I\cap I(J))$ and same for $\sigma$.

(b) Assume $I$ intersects $J$ but not $I(J)$. Since by assumption $I$ is not contained in any triadic subinterval of $J$ of length $|I(J)|=3^{1-k}|J|$, we deduce that it must contain an endpoint of some triadic subinterval of $J$ of length $3^{1-k}|J|$, which is by definition an interval in $\mathbf{K}$. Therefore, if $|I|<\left(\frac{1}{3}-3^{-k}\right)|I(J)|$, then Remark \ref{r: vanish} implies immediately that $w,\sigma$ vanish on $I$.

Assume now that $|I|\geq\left(\frac{1}{3}-3^{-k}\right)|I(J)|$. Let $\mathcal{I}$ be the family of all triadic subintervals of $J$ of length $|I(J)|=3^{1-k}|J|$ intersecting $I$, and set $N:=\#\mathcal{I}$. Choose $K'\in\mathcal{I}$. It is then clear that
\begin{equation*}
\max(1,(N-2))|K'|=\max(1,(N-2))|I(J)|\lesssim|I|,
\end{equation*}
and it also follows from Lemma \ref{average_estimates_triadic} that
\begin{equation*}
w(I)\leq Nw(K'),~\sigma(I)\leq N\sigma(K'),
\end{equation*}
implying that
\begin{equation*}
\La w\Ra\ci{I}\lesssim\La w\Ra\ci{K'}\sim\La w\Ra\ci{K},~\La \sigma\Ra\ci{I}\lesssim\La \sigma\Ra\ci{K'}\sim_{p}\La \sigma\Ra\ci{K}.
\end{equation*}

(c) Assume that $I$ intersects both $J$ and $I(J)$. In view of Remark \ref{r: vanish} and the facts that $\La \sigma\Ra\ci{K}\sim_{p} 3^{-k}\La\sigma\Ra\ci{I(J)}$ and $|I|\leq |K|=3^{k}|I(J)|$, it is easy to see that it suffices to prove that
\begin{equation*}
\La w\Ra\ci{I\cap J}\lesssim\La w\Ra\ci{I(J)},~\La\sigma\Ra\ci{I\cap J}\lesssim_{p}\La\sigma\Ra\ci{K}.
\end{equation*}
Note that $I\cap J$ contains the common endpoint of $J,I(J)$. Thus, if $|I\cap J|\leq \left(\frac{1}{3}-3^{-k}\right)|I(J)|$, then $w,\sigma$ vanish on $I\cap J$ and we have nothing to show. If $I\cap J$ is not contained in any triadic subinterval of $J$ of length $3^{1-k}|J|$, then we are done by (b).

Assume now that $|I\cap J|>\left(\frac{1}{3}-3^{-k}\right)|I(J)|$, and that $I\cap J$ is contained in a triadic subinterval $K'$ of $J$ of length $3^{1-k}|J|$. Then $w(I\cap J)\leq w(K')$, $\sigma(I\cap J)\leq \sigma(K')$, and $|I\cap J|\gtrsim |K'|$, thus we are done by Lemma \ref{average_estimates_triadic}.
\end{proof}

\begin{rem}\label{A_p}
Lemma \ref{average_estimates} shows in particular that
\begin{equation}
\label{A_p-estimate}
\sup\ci{I}\La w\Ra\ci{I}^{p-1}\La\sigma\Ra\ci{I},~\sup\ci{I}\La \sigma\Ra\ci{I}^{p'-1}\La w\Ra\ci{I}\lesssim_{p}1,
\end{equation}
where supremum is taken over all subintervals $I$ of $[0,1)$. Of course, estimate \eqref{A_p-estimate} also follows immediately from the observation due to Reguera--Scurry \cite{reguera-scurry} (which can be deduced from Lemmas \ref{average_estimates_triadic} and \ref{average_estimates} as well) that $Mw\lesssim w$ on $\lbrace w>0\rbrace$, so $\sigma\sim\frac{w}{(Mw)^{p}}$.
\end{rem}

In the sequel we will make use of the following comparison lemma.

\begin{lm}\label{comparison_contain_endpoint}
Let $K\in\mathbf{K}$, and let $L$ be a subinterval of $K$ sharing an endpoint with $K$. Then, for all intervals $L'\subseteq L$ sharing an endpoint with $K$, we have $\La w\Ra\ci{L'}\lesssim\La w\Ra\ci{L}$ and $\La \sigma\Ra\ci{L'}\lesssim\La \sigma\Ra\ci{L}$.
\end{lm}
\begin{proof}
The result is clear if $L'$ intersects neither $J$ nor $I(J)$. If $L'$ intersects $J\cup I(J)$, then $|L'|\gtrsim |K|\geq|L|$ and thus the desired result is again clear.
\end{proof}

We conclude this subsection with the following observation.

\begin{lm}\label{packing_K}
Let $K\in\mathbf{K}$. Then, we have
\begin{equation*}
\sum_{\substack{K'\in\mathbf{K}\\K'\subseteq K}}w(K')\sim 3^{k}w(K),~\sum_{\substack{K'\in\mathbf{K}\\K'\subseteq K}}\sigma(K')\sim_{p}\sigma(K).
\end{equation*}
\end{lm}
\begin{proof}
By Lemma \ref{average_estimates_triadic} we have
\begin{align*}
\sum_{\substack{K'\in\mathbf{K}\\K'\subseteq K}}w(K')\sim\sum_{\substack{K'\in\mathbf{K}\\K'\subseteq K}}3^{k}w(I((K')^m))=3^{k}w(K)
\end{align*}
and
\begin{align*}
\sum_{\substack{K'\in\mathbf{K}\\K'\subseteq K}}\sigma(K')\sim_{p}\sum_{\substack{K'\in\mathbf{K}\\K'\subseteq K}}\sigma(I((K')^m))=\sigma(K).
\end{align*}
\end{proof}

\section{Two-weight estimates for generalized sparse operators}
\label{s: sparse-square-estimates}

In this section we obtain two-weight estimates for sparse $p$-functions and sparse $p'$-functions, with respect to the weights introduced in Section \ref{s: reguera-thiele}. After rescaling these weights, and applying a ``direct sum of singularities'' type argument following the one used by Reguera--Scurry in  \cite{reguera-scurry}, we obtain a proof of Proposition \ref{p: main}.

\subsection{Sparse families} \label{s: sparse definitions}

In this subsection we fix notation and terminology regarding sparse families, following Subsection 2.1 of \cite{convex_body}.

\begin{df} \label{df: sparse} Let $0<\eta<1$. A family $\cS$ of cubes in $\R^d$ is said to be (weakly) \emph{$\eta$-sparse}, if there exists a family $\lbrace E(I):~I\in\cS\rbrace$ of pairwise disjoint measurable subsets of $\R^d$, such that $E(I)\subseteq I$ and $|E(I)|\geq(1-\eta)|I|$, for all $I\in\cS$.
\end{df}

An alternative definition of sparse families is often more useful for the purpose of estimating sparse operators.
 
 \begin{df} \label{df: martingale sparse} Let $\cD$ be a (nonhomogeneous) grid of cubes in $\R^d$, in the sense of \cite{culiuc} and \cite{treil}, that is one can write
\begin{equation*}
\cD=\bigcup_{k\in\Z}\cD_k,
\end{equation*}
where for all $k\in\Z$, $\cD_k$ is an at most countable family of pairwise disjoint cubes in $\R^d$ covering $\R^d$, and for all $k,l\in\Z$ with $k<l$, for all $Q\in\cD_k$ and for all $R\in\cD_l$, there holds either $R\subseteq Q$ or $R\cap Q=\emptyset$. Note that for all $Q,R\in\cD$ we have either $Q\subseteq R$ or $R\subsetneq Q$ or $Q\cap R=\emptyset$. Let $0<\e<1$. A subfamily $\cS$ of $\cD$ is said to be \emph{martingale $\e$-sparse} if
\begin{equation*}
\sum_{Q\in\text{ch}\ci{\cS}(R)}|Q|\leq\e|R|,
\end{equation*}
where $\text{ch}\ci{\cS}(R)$ is the family of all maximal cubes in $\cS$ that are strictly contained in $R$, for all $R\in\cS$.
\end{df}

Note that if $\mathcal{S}$ is a martingale $\e$-sparse family of cubes in $\R^d$, then $\cS$ is $\e$-sparse in the first sense, since one can just define $E(R):=R\setminus\left(\bigcup_{Q\in\text{ch}\ci{\cS}(R)}Q\right)$, for all $R\in\cS$, so in particular for all cubes $L$ in $\R^d$ we have
\begin{equation}
\label{sparse packing}
\sum_{\substack{Q\in\mathcal{S}\\Q\subseteq L}}|Q|\leq\frac{|L|}{1-\e}.
\end{equation}
Although Definition \ref{df: martingale sparse} seems more restrictive than Definition \ref{df: sparse}, as it is explained in \cite{convex_body} estimates for sparse operators over sparse families in the first sense can be reduced to estimates for sparse operators over sparse families in the second sense. For reasons of completeness, we include the details of this reduction in the appendix.

Let now $\mathcal{S}$ be a martingale $\e$-sparse family of cubes in $\R^d$ for some $0<\e<1$. Then, for all $Q,Q'\in\cS$ we have $Q\subseteq Q'$ or $Q'\subsetneq Q$ or $Q\cap Q'=\emptyset$. Moreover, for all $Q,Q'\in\mathcal{S}$ with $Q'\subsetneq Q$ we have $|Q'|\leq\e|Q|$. Therefore, if $\mathcal{S}'\subseteq\cS$ is a chain (i.e. a subfamily of $\cS$ linearly ordered with respect to containment) such that there exists $c>0$ with $|Q|\geq c$, for all $Q\in\cS'$, then we have
\begin{equation}
\label{estimate_chain}
\sum_{Q\in\mathcal{S}'}\frac{1}{|Q|^{p}}\leq\frac{1}{c^{p}}\sum_{n=0}^{\infty}\e^{pn}=\frac{1}{c^{p}(1-\e^{p})}\leq\frac{1}{c^{p}(1-\e)}.
\end{equation}
As we will see below, estimate \eqref{estimate_chain} will be enough to deal with estimates for sparse $p$-functions over sparse families of triadic intervals, but in order to deal with general sparse families, a refined version of estimate \eqref{estimate_chain} will be necessary.
\begin{lm}\label{restricted_sparse_packing}
Let $\mathcal{S}$ be a martingale $\e$-sparse family of cubes in $\R^d$. Assume that all cubes in  $\mathcal{S}$ are contained in a cube $L$. Let $E$ be a measurable subset of $\R^d$. Then, there holds
\begin{equation*}
\sum_{Q\in\mathcal{S}}\left(\frac{|Q\cap E|}{|Q|}\right)^{p+1}|Q|\lesssim_{p}\frac{|L\cap E|}{1-\e}.
\end{equation*}
\end{lm}

 In order to prove Lemma \ref{restricted_sparse_packing}, we will use the Carleson Embedding Theorem, in the version stated in \cite[Lemma 5.1]{convex_body}.

\begin{lm}[Carleson Embedding Theorem]\label{l: carleson}
Let $\mu$ be a Radon measure on $\R^d$, and let $\cD$ be a grid of cubes in $\R^d$. Let $\lbrace a\ci{Q}:~Q\in\cD\rbrace$ be a collection of nonnegative real numbers such that
\begin{equation*}
\sum_{\substack{Q\in\cD\\Q\subseteq R}}a\ci{Q}\mu(Q)\leq A\mu(R),~\forall J\in\cD,
\end{equation*}
for some $A>0$. Then, for all measurable functions $f\geq0$ on $\R^d$ and for all $1<p<\infty$, there holds
\begin{equation}
\label{carleson}
\sum_{Q\in\cD}\left(\frac{1}{\mu(Q)}\int_{Q}f(x)d\mu(x)\right)^{p}a\ci{Q}\mu(Q)\leq(p')^{p}A\Vert f\Vert\ci{L^{p}(\mu)}^p.
\end{equation}
\end{lm}

A proof of that version of the Carleson Embedding Theorem can be found in \cite{treil}. Let us note that the exact constant $C(p)$ appearing in the right-hand side of \eqref{carleson} is not important for our purposes.

\begin{proof}[Proof (of Lemma \ref{restricted_sparse_packing})]
Applying Lemma \ref{l: carleson} for the function $f:=1\ci{L\cap E}$, for the exponent $p+1$ and with $\mu$ being Lebesgue measure on $\R^d$, and using \eqref{sparse packing}, we obtain
\begin{align*}
\sum_{Q\in\mathcal{S}}\left(\frac{|Q\cap E|}{|Q|}\right)^{p+1}|Q|=\sum_{Q\in\cS}(\La f\Ra\ci{Q})^{p+1}|Q|\lesssim_{p}\frac{\Vert f\Vert\ci{L^{p+1}(\R^d)}^{p+1}}{1-\e}=\frac{|L\cap E|}{1-\e}.
\end{align*}
\end{proof}

\subsection{The main estimates} \label{s: main sparse estimates}

Here we state the main estimates that lead to a proof of Proposition \ref{p: main}. Fix $1<p<\infty$. For all positive integers $k>3000$, we denote by $w_k,\sigma_k$ the weights on $[0,1)$ constructed in Section \ref{s: reguera-thiele} for these $k,p$, following the notation used by Reguera--Scurry \cite{reguera-scurry}.

\begin{prop}\label{p: local_testing_condition}
Let $0<\e<1$, and let $\cS$ be a martingale $\e$-sparse family of subintervals of $[0,1)$. Fix a positive integer $k>3000$. Then, we have
\begin{equation}
\label{testing_conditions}
\sum_{\substack{I\in\mathcal{S}\\I\subseteq L}}(\La w_k\Ra\ci{I})^{p}\La\sigma_k\Ra\ci{I}|I|\lesssim_{p} k\frac{w_k(L)}{1-\e},\qquad \sum_{\substack{I\in\mathcal{S}\\I\subseteq L}}(\La \sigma_k\Ra\ci{I})^{p'}\La w_k\Ra\ci{I}|I|\lesssim_{p} k\frac{\sigma_k(L)}{1-\e},
\end{equation}
for any subinterval $L$ of $[0,1)$.
\end{prop}

The proof of Proposition \ref{p: local_testing_condition} is postponed to Subsection \ref{s: proof of local testing}. Note that the estimates in \eqref{testing_conditions} blow up as $k\rightarrow\infty$. To rectify this, we pick some
\begin{equation*}
r\in\left(\max\left(1,\frac{1}{p-1}\right),p'\right),
\end{equation*}
where we recall that $p':=\frac{p}{p-1}$, and we define a rescaled version $\wt{w}_k$ of $w_{k}$ by
\begin{equation*}
\wt{w}_{k}:=k^{-r}w_{k}.
\end{equation*}
Since $pr>r+1$ and $r>1$, it is easy to see that for any martingale $\e$-sparse family of subintervals of $[0,1)$ and any subinterval $L$ of $[0,1)$ there holds
\begin{equation}
\label{rescaled local testing conditions}
\sum_{\substack{I\in\mathcal{S}\\I\subseteq L}}(\La \wt{w}_k\Ra\ci{I})^{p}\La\sigma_k\Ra\ci{I}|I|\lesssim_{p}\frac{\wt{w}_k(L)}{1-\e},\qquad \sum_{\substack{I\in\mathcal{S}\\I\subseteq L}}(\La \sigma_k\Ra\ci{I})^{p'}\La\wt{w}_k\Ra\ci{I}|I|\lesssim_{p}\frac{\sigma_k(L)}{1-\e}.
\end{equation}
These estimates are uniform with respect to $k$. Following Reguera--Scurry \cite{reguera-scurry} we take the ``direct sum of singularities'', defining the weights $\wt{w},\sigma$ on $\R$ given by
\begin{equation*}
\wt{w}(x):=\sum_{k=4000}^{\infty}\wt{w}_k(x-9^k),\qquad\sigma(x):=\sum_{k=4000}^{\infty}\sigma_{k}(x-9^k),\qquad x\in\R.
\end{equation*}
It is not hard to see, and is explained in detail in Subsection \ref{s: direct sum}, that for all martingale $\e$-sparse families $\cS$ of intervals in $\R$ there holds
\begin{equation}
\label{global testing conditions}
\sum_{\substack{I\in\mathcal{S}\\I\subseteq L}}(\La \wt{w}\Ra\ci{I})^{p}\La\sigma\Ra\ci{I}|I|\lesssim_{p}\frac{\wt{w}(L)}{1-\e},\qquad \sum_{\substack{I\in\mathcal{S}\\I\subseteq L}}(\La \sigma\Ra\ci{I})^{p'}\La\wt{w}\Ra\ci{I}|I|\lesssim_{p}\frac{\sigma(L)}{1-\e}.
\end{equation}
The last estimates are referred to as (Sawyer-type) \emph{testing conditions}. To extend them to full bounds for the operators of interest, we will use the following special case of a result due to Culiuc \cite{culiuc}.

\begin{thm}[Culiuc \cite{culiuc}]\label{thm: testing to general}
Let $\cD$ be a grid of cubes in $\R^d$. Let $\lbrace a\ci{Q}\rbrace\ci{Q\in\cD}$ be a family of nonnegative measurable functions on $\R^d$. Let $1<p_1\leq p_2<\infty$. Consider the operator $T$ given by
\begin{equation*}
Tf:=\left(\sum_{Q\in\cD}{(|\La f\Ra\ci{Q}|)}^{p_2}a\ci{Q}1\ci{Q}\right)^{1/p_2},\qquad\forall f\in L^1\ti{loc}(\R^d),
\end{equation*}
and for all $L\in\cD$, consider the localized truncation $T\ci{L}$ of $T$ given by
\begin{equation*}
T\ci{L}f:=\bigg(\sum_{\substack{Q\in\cD\\Q\subseteq L}}{(|\La f\Ra\ci{Q}|)}^{p_2}a\ci{Q}1\ci{Q}\bigg)^{1/p_2},\qquad\forall f\in L^1\ti{loc}(\R^d).
\end{equation*} 
Let $u,v$ be weights on $\R^d$, such that for some $A>0$ there holds
\begin{equation*}
\Vert T\ci{L}(1\ci{L}v)\Vert\ci{L^{p_1}(u)}\leq A\Vert 1\ci{L}\Vert\ci{L^{p_1}(v)},\qquad\forall L\in\cD.
\end{equation*}
Then, there exists a constant $C(p_1)>0$ such that $\Vert T(\fdot v)\Vert\ci{L^{p_1}(v)\rightarrow L^{p_1}(u)}\leq C(p_1)A$.
\end{thm}

In view of Theorem \ref{thm: testing to general}, the estimates in \eqref{global testing conditions} imply that
\begin{equation*}
\Vert \mathcal{A}\ci{\cS,p}(\fdot\wt{w})\Vert\ci{L^{p}(\wt{w})\rightarrow L^{p}(\sigma)}\lesssim_{p}\frac{1}{(1-\e)^{1/p}},\qquad\Vert \mathcal{A}\ci{\cS,p'}(\fdot\sigma)\Vert\ci{L^{p'}(\sigma)\rightarrow L^{p'}(\wt{w})}\lesssim_{p}\frac{1}{(1-\e)^{1/p'}},
\end{equation*}
for any martingale $\e$-sparse family $\cS$ of intervals in $\R$. The reduction from general sparse families to martingale sparse families described in the Appendix allows then to conclude that for any $0<\eta<1$ and for any $\eta$-sparse family $\cS$ of intervals in $\R$ there holds
\begin{equation*}
\Vert \mathcal{A}\ci{\cS,p}(\fdot\wt{w})\Vert\ci{L^{p}(\wt{w})\rightarrow L^{p}(\sigma)},~\Vert \mathcal{A}\ci{\cS,p'}(\fdot\sigma)\Vert\ci{L^{p'}(\sigma)\rightarrow L^{p'}(\wt{w})}\leq C=C(\eta,p).
\end{equation*}

Recall also estimate \eqref{rescaled norm Hilbert transform estimate}
\begin{equation*}
 \Vert H(1\ci{[0,1)}\wt{w}_{k})\Vert\ci{L^{p}(\sigma_{k})}\gtrsim k^{1-(r/p')}\Vert 1\ci{[0,1)}\Vert\ci{L^{p}(\wt{w}_k)},\qquad\forall k>3000.
\end{equation*}
Coupled with translation invariance, it yields
\begin{align*}
\Vert H(1\ci{I_{k}}\wt{w})\Vert\ci{L^p(\sigma)}\gtrsim k^{1-(r/p')}\Vert 1\ci{I_{k}}\Vert\ci{L^p(\wt{w})},
\end{align*}
where $I_k:=[9^k,9^k+1)$, for all $k\geq 4000$. Since $1-(r/p')>0$, we deduce
\begin{equation*}
\Vert H(\fdot\wt{w})\Vert\ci{L^p(\wt{w})\rightarrow L^p(\sigma)}=\infty.
\end{equation*}
An application of the closed graph theorem, coupled with the facts that $\wt{w}\in L^1(\R)$, so $L^{p}(\wt{w})\subseteq L^1(\wt{w})$, and that the linear operator $H:L^{1}(\R)\rightarrow L^{1,\infty}(\R)$ is bounded, implies that there exists $f\in L^{p}(\wt{w})$ with $H(f\wt{w})\notin L^{p}(\sigma)$.

\subsection{Verification of local testing conditions}\label{s: proof of local testing}

The proof of Proposition \ref{p: local_testing_condition} will be accomplished in several steps.  We fix a positive integer $k>3000$, $0<\e<1$, and an $\e$-martingale sparse family $\cS$ of subintervals of $[0,1)$. To simplify the notation, we denote $w_k,\sigma_k$ by $w,\sigma$ respectively. Note that by applying the Monotone Convergence Theorem (for series) we can without loss of generality assume that the sparse family $\cS$ is finite (as long as no estimate depends on cardinality), and we will be doing so in the sequel. By interval we mean a subset of $\R$ of the form $[a,b)$, where $a,b\in\R$, $a<b$.

\subsubsection{Triadic case} Here we give a simpler proof of Proposition \ref{p: local_testing_condition} for the case that all intervals are triadic. In this case, it is actually possible to prove a better estimate. The triadic case already shows some of the main difficulties that will arise in the general case and what strategy one should follow to deal with them.

\begin{prop}\label{local_testing_triadic}
Assume that the sparse family $\mathcal{S}$ consists of triadic intervals. Let $L$ by any triadic subinterval of $[0,1)$. Then
\begin{equation}
\label{testing_conditions_triadic}
\sum_{\substack{I\in\mathcal{S}\\I\subseteq L}}(\La w\Ra\ci{I})^{p}\La\sigma\Ra\ci{I}|I|\lesssim_{p} \frac{w(L)}{1-\e},\qquad \sum_{\substack{I\in\mathcal{S}\\I\subseteq L}}(\La \sigma\Ra\ci{I})^{p'}\La w\Ra\ci{I}|I|\lesssim_{p} \frac{\sigma(L)}{1-\e}.
\end{equation}
\end{prop}

The following lemma establishes Proposition \ref{local_testing_triadic} in an important special case.

\begin{lm}\label{testing_condition_K_triadic}
Assume that the sparse family $\mathcal{S}$ consists of triadic intervals. Let $K\in\mathbf{K}$. Set $J:=K^{m}$. Let $S(J)$ be the triadic child of $K$ containing $I(J)$. Set
\begin{equation*}
\mathcal{S}\ci{K}^{1}:=\lbrace I\in\mathcal{S}:~I=K,\text{ or }I\subseteq J\text{ and }I\text{ is not contained in any } K'\in\textup{ch}\ci{\mathbf{K}}(K)\rbrace,
\end{equation*}
\begin{equation*}
\cS\ci{K}^2:=\lbrace I\in\cS:~ I(J)\subsetneq I\subseteq S(J)\rbrace.
\end{equation*}

(a) There holds
\begin{equation*}
\sum_{I\in\mathcal{S}\ci{K}^m}(\La w\Ra\ci{I})^{p}\La \sigma\Ra\ci{I}|I|\lesssim_{p}\frac{3^{-k}w(K)}{1-\e},\qquad\sum_{I\in\mathcal{S}\ci{K}^m}(\La \sigma\Ra\ci{I})^{p'}\La w\Ra\ci{I}|I|\lesssim_{p}\frac{\sigma(K)}{1-\e},\qquad m=1,2.
\end{equation*}

(b) The testing conditions in \eqref{testing_conditions_triadic} hold in the case $L=K$.
\end{lm}

\begin{proof}
(a) By Lemma \ref{average_estimates_triadic} we have
\begin{align*}
\sum_{I\in\mathcal{S}^{1}\ci{K}}(\La w\Ra\ci{I})^{p}\La \sigma\Ra\ci{I}|I|\sim_{p}3^{-k}\sum_{I\in\mathcal{S}^{1}\ci{K}}\La w\Ra\ci{K}|I|\leq\frac{3^{-k}\La w\Ra\ci{K}|J|}{1-\e}=\frac{3^{-k-1}w(K)}{1-\e}.
\end{align*}
Moreover, all intervals in $\cS^2\ci{K}$ contain $|I(J)|$, thus by \eqref{estimate_chain} we deduce
\begin{align*}
\sum_{I\in\cS^2\ci{K}}(\La w\Ra\ci{I})^{p}\La \sigma\Ra\ci{I}|I|&=\sum_{I\in\cS\ci{K}^2}\left(\frac{|I(J)|}{|I|}\right)^{p+1}(\La w\Ra\ci{I(J)})^{p}\La \sigma\Ra\ci{I(J)}|I|\leq\frac{\La w\Ra\ci{I(J)}|I(J)|}{1-\e}\\&
\sim\frac{\La w\Ra\ci{K}3^{-k}|K|}{1-\e}=
\frac{3^{-k}w(K)}{1-\e}.
\end{align*}
The second estimate is proved similarly, recalling that $\La\sigma\Ra\ci{I(J)}\sim_{p}3^{k}\La \sigma\Ra\ci{K}$.

(b) Clearly
\begin{align*}
\sum_{\substack{I\in\mathcal{S}\\I\subseteq K}}(\La w\Ra\ci{I})^{p}\La \sigma\Ra\ci{I}|I|&=\sum_{\substack{K'\in\mathbf{K}\\K'\subseteq K}}\bigg(\sum_{\substack{I\in\cS\\I\subseteq I((K')^m)}}(\La w\Ra\ci{I})^{p}\La \sigma\Ra\ci{I}|I|+\sum_{m=1}^{2}\sum_{I\in\mathcal{S}\ci{K'}^{m}}(\La w\Ra\ci{I})^{p}\La \sigma\Ra\ci{I}|I|\bigg)
\end{align*}

By (a) and Lemma \ref{packing_K} we deduce
\begin{align*}
\sum_{\substack{K'\in\mathbf{K}\\K'\subseteq K}}\sum_{m=1}^{2}\sum_{I\in\mathcal{S}\ci{K'}^{m}}(\La w\Ra\ci{I})^{p}\La \sigma\Ra\ci{I}|I|\lesssim_{p}\frac{3^{-k}}{1-\e}\sum_{\substack{K'\in\mathbf{K}\\K'\subseteq K}}w(K')\sim\frac{w(K)}{1-\e}.
\end{align*}
Moreover, by Lemma \ref{average_estimates_triadic} and \eqref{sparse packing} we have
\begin{equation*}
\sum_{\substack{K'\in\mathbf{K}\\K'\subseteq K}}\sum_{\substack{I\in\cS\\I\subseteq I((K')^m)}}(\La w\Ra\ci{I})^{p}\La \sigma\Ra\ci{I}|I|\leq
\frac{1}{1-\e}\sum_{\substack{K'\in\mathbf{K}\\K'\subseteq K}}w(I((K')^m))=\frac{w(K)}{1-\e},
\end{equation*}
concluding the proof of the first estimate. The second estimate is proved similarly.
\end{proof}

We now prove Proposition \ref{local_testing_triadic}.

\begin{proof}[Proof (of Proposition \ref{local_testing_triadic})]
Let $K$ be the smallest interval in $\mathbf{K}$ containing $L$. If $L=K$, then we are done by Lemma \ref{testing_condition_K_triadic}. Assume now that $L\neq K$. Set $J:=K^{m}$. Since $w,\sigma$ are constant on $I(J)$, by \eqref{sparse packing} we have
\begin{equation*}
\sum_{\substack{I\in\cS\\I\subseteq L\cap I(J)}}(\La w\Ra\ci{I})^{p}\La\sigma\Ra\ci{I}|I|\leq\frac{\La w\Ra\ci{I(J)}|L\cap I(J)|}{1-\e}=\frac{w(L\cap I(J))}{1-\e}\leq\frac{w(L)}{1-\e}.
\end{equation*}
If there are intervals in $\cS$ that are contained in $L$ and strictly contain $I(J)$, then $I(J)\subseteq L$, and these intervals have length greater than or equal to $I(J)$, form a chain and do not intersect $J$, therefore by \eqref{estimate_chain} we obtain
\begin{align*}
\sum_{\substack{I\in\cS\\ I(J)\subsetneq I\subseteq L}}(\La w\Ra\ci{I})^{p}\La \sigma\Ra\ci{I}|I|
&=\sum_{\substack{I\in\cS\\ I(J)\subsetneq I\subseteq L}}\left(\frac{|I(J)|}{|I|}\right)^{p+1}(\La w\Ra\ci{I(J)})^{p}\La \sigma\Ra\ci{I(J)}|I|\leq\frac{w(I(J))}{1-\e}\leq\frac{w(L)}{1-\e}.
\end{align*}
Finally, assume $L\cap J\neq\emptyset$. Then, by Lemma \ref{average_estimates_triadic} we have $\La w\Ra\ci{L}\sim\La w\Ra\ci{K}$ and $\La \sigma\Ra\ci{L}\sim_{p}\La\sigma\Ra\ci{K}$. Let $\wt{A}\ci{J}$ be the family of all triadic subintervals of $J$ of length $3^{1-k}|J|$ contained in $L$. Then, by Lemma \ref{average_estimates_triadic} and \eqref{sparse packing} we have
\begin{align*}
\sum_{\substack{I\in\mathcal{S}\\I\subseteq L\cap J\\I\nsubseteq K',~\forall K'\in\wt{A}\ci{J}}}(\La w\Ra\ci{I})^{p}\La \sigma\Ra\ci{I}|I|\lesssim_{p}
\sum_{\substack{I\in\mathcal{S}\\I\subseteq L\cap J\\I\nsubseteq K',~\forall K'\in\wt{A}\ci{J}}}\La w\Ra\ci{L}|I|\lesssim_{p}\La w\Ra\ci{L}\frac{|L|}{1-\e}=\frac{w(L)}{1-\e}.
\end{align*}
Moreover, by Lemma \ref{testing_condition_K_triadic} we have
\begin{align*}
\sum_{K'\in\wt{A}\ci{J}}\sum_{\substack{I\in\mathcal{S}\\I\subseteq K'}}(\La w\Ra\ci{I})^{p}\La \sigma\Ra\ci{I}|I|\lesssim_{p}\frac{1}{1-\e}\sum_{K'\in\wt{A}\ci{J}}w(K')=\frac{w(L\cap J)}{1-\e}\leq\frac{w(L)}{1-\e},
\end{align*}
concluding the proof of the first estimate. The second one is proved similarly.
\end{proof}

\subsubsection{General case} Here we prove Proposition \ref{p: local_testing_condition} in the general case. The main strategy will be the same with the one in the triadic case.

We begin by establishing restricted versions of the testing conditions in  \eqref{testing_conditions}, for general intervals $L$.

\begin{lm}\label{testing_intersect_only_I(J)_w}
Let $L$ be any subinterval of $[0,1)$, and let $K\in\mathbf{K}$ with $L\subseteq K$. Set $J:=K^{m}$. Let $\cS'$ be the family of all intervals $I$ in $\cS$ contained in $L$ such that $w,\sigma$ vanish on $I\cap J$. Then, there holds
\begin{equation*}
\sum_{I\in\mathcal{S}'}(\La w\Ra\ci{I})^{p}\La\sigma\Ra\ci{I}|I|\lesssim_{p} \frac{w(L)}{1-\e},\qquad\sum_{I\in\mathcal{S}'}(\La \sigma\Ra\ci{I})^{p'}\La w\Ra\ci{I}|I|\lesssim_{p} \frac{\sigma(L)}{1-\e}.
\end{equation*}
\end{lm}
\begin{proof}
Since $w,\sigma$ are constant on $I(J)$, by Lemma \ref{restricted_sparse_packing} we have
\begin{align*}
\sum_{I\in\mathcal{S}'}(\La w\Ra\ci{I})^{p}\La\sigma\Ra\ci{I}|I|&
=\sum_{I\in\mathcal{S}'}\left(\frac{|I\cap I(J)|}{|I|}\right)^{p+1}(\La w\Ra\ci{I(J)})^{p}\La\sigma\Ra\ci{I(J)}|I|\lesssim_{p}
\La w\Ra\ci{I(J)}\frac{|L\cap I(J)|}{1-\e}\\
&=\frac{w(L\cap I(J))}{1-\e}\leq\frac{w(L)}{1-\e}.
\end{align*}
The second estimate is proved similarly.
\end{proof}

\begin{lm}\label{partial_testing_w}
Let $K\in\mathbf{K}$, and $L$ be any subinterval of $K$. Set $J:=K^{m}$.

(a) There holds
\begin{equation*}
\sum_{\substack{I\in\mathcal{S},~I\subseteq L\\|I|\geq\left(\frac{1}{3}-3^{-k}\right)|I(J)|}}(\La w\Ra\ci{I})^{p}\La\sigma\Ra\ci{I}|I|\lesssim_{p}\frac{kw(L)}{1-\e},\qquad\sum_{\substack{I\in\mathcal{S},~I\subseteq L\\I|\geq\left(\frac{1}{3}-3^{-k}\right)|I(J)|}}(\La \sigma\Ra\ci{I})^{p'}\La w\Ra\ci{I}|I|\lesssim_{p}\frac{k\sigma(L)}{1-\e}
\end{equation*}

(b) There holds
\begin{equation*}
\sum_{\substack{I\in\mathcal{S},~I\subseteq L\\~I\nsubseteq K',~\forall K'\in\textup{ch}\ci{\mathbf{K}}(K)}}(\La w\Ra\ci{I})^{p}\La\sigma\Ra\ci{I}|I|\lesssim_{p}\frac{kw(L)}{1-\e},\qquad\sum_{\substack{I\in\mathcal{S},~I\subseteq L\\~I\nsubseteq K',~\forall K'\in\textup{ch}\ci{\mathbf{K}}(K)}}(\La \sigma\Ra\ci{I})^{p'}\La w\Ra\ci{I}|I|\lesssim_{p}\frac{k\sigma(L)}{1-\e}.
\end{equation*}
\end{lm}
\begin{proof}
(a) Let $\mathcal{S}\ci{L}$ be the family of intervals in $\mathcal{S}$ contained in $L$ of length at least $\left(\frac{1}{3}-3^{-k}\right)|I(J)|$. Let $N$ be the maximum length of a  chain (with respect to containment) in $\mathcal{S}\ci{L}$. Set
\begin{equation*}
\cS\ci{L}^{1}:=\lbrace\text{maximal intervals in }\mathcal{S}\ci{L}\rbrace,
\end{equation*}
and define inductively
\begin{equation*}
\cS\ci{L}^{n}:=\left\lbrace\text{maximal intervals in }\mathcal{S}\ci{L}\setminus\left(\bigcup_{r=1}^{n-1}\cS\ci{L}^{r}\right)\right\rbrace,~n=2,\ldots,N.
\end{equation*}
Notice that $\mathcal{S}\ci{L}=\bigcup_{n=1}^{N}\mathcal{S}\ci{L}^{n}$. In particular, for all $n=2,\ldots,N$, for all $I\in\mathcal{S}\ci{L}^{n}$, there exists $I'\in\mathcal{S}\ci{L}^{n-1}$ with $I\subsetneq I'$. Thus, picking some $I\ci{N}\in\mathcal{S}\ci{L}^{N}$, one can find $I\ci{1}\in\mathcal{S}\ci{L}^{1}$ with $I\ci{N}\subsetneq I_{1}$, therefore
\begin{equation*}
\e^{N-1}3^{k}|I(J)|=\e^{N-1}|K|\geq\e^{N-1}|L|\geq\e^{N-1}|I_{1}|\geq|I\ci{N}|\geq\left(\frac{1}{3}-3^{-k}\right)|I(J)|\geq\frac{1}{4}|I(J)|,
\end{equation*}
therefore
\begin{equation*}
N\leq\frac{k\log 3-\log 4}{\log(1/\e)}+1\lesssim\frac{k}{1-\e}.
\end{equation*}
It follows by Remark \ref{A_p} that
\begin{align*}
\sum_{\substack{I\in\mathcal{S}\\I\subseteq L,~|I|\geq\left(\frac{1}{3}-3^{-k}\right)|I(J)|}}(\La w\Ra\ci{I})^{p}\La\sigma\Ra\ci{I}|I|&\lesssim_{p}\sum_{I\in\mathcal{S}\ci{L}}\La w\Ra\ci{I}|I|
=\sum_{n=1}^{N}w\left(\bigcup\mathcal{S}\ci{L}^{n}\right)\leq
Nw(L)\lesssim\frac{kw(L)}{1-\e},
\end{align*}
concluding the proof of the first estimate. The second one is proved similarly.

(b) In view of (a), we only have to prove that
\begin{equation}
\label{inter_estimate}
\sum_{I\in\mathcal{S}'}(\La w\Ra\ci{I})^{p}\La\sigma\Ra\ci{I}|I|\lesssim_{p}\frac{w(L)}{1-\e},\qquad \sum_{I\in\mathcal{S}'}(\La \sigma\Ra\ci{I})^{p'}\La w\Ra\ci{I}|I|\lesssim_{p}\frac{\sigma(L)}{1-\e},
\end{equation}
where
\begin{equation*}
\mathcal{S}':=\left\lbrace I\in\mathcal{S}:~I\subseteq L,~|I|<\left(\frac{1}{3}-3^{-k}\right)|I(J)|,~I\nsubseteq K',~\forall K'\in\text{ch}\ci{\mathbf{K}}(K)\right\rbrace.
\end{equation*}
Recall that $\text{ch}\ci{\mathbf{K}}(K)$ coincides with the family of all triadic subintervals of $J$ of length $|I(J)|=3^{1-k}|J|$. Therefore, if an interval $I\in\mathcal{S}'$ intersects $J$, then it must contain an endpoint of some interval in $\text{ch}\ci{\mathbf{K}}(K)$, and since $|I|<\left(\frac{1}{3}-3^{-k}\right)|I(J)|$ it follows immediately from Remark \ref{r: vanish} that $w,\sigma$ vanish on $I\cap J$. Thus, \eqref{inter_estimate} follows immediately from Lemma \ref{testing_intersect_only_I(J)_w}.
\end{proof}

Proposition \ref{testing_condition_K_w} and Lemma \ref{testing_contains_endpoint_K_w} below establish the testing conditions \eqref{testing_conditions} for special classes of intervals $L$.

\begin{prop}\label{testing_condition_K_w}
Let $K\in\mathbf{K}$. Set $J:=K^{m}$, and define the families
\begin{equation*}
\mathcal{S}^{1}\ci{K}:=\lbrace I\in\mathcal{S}:~I\subseteq K\setminus I(J),~I\text{ is not contained in any }K'\in\textup{ch}\ci{\mathbf{K}}(K)\rbrace,
\end{equation*}
\begin{equation*}
\mathcal{S}^{2}\ci{K}:=\lbrace I\in\mathcal{S}:~I\subseteq K,~I\text{ intersects }I(J)\text{ but not }J\rbrace,
\end{equation*}
\begin{equation*}
\mathcal{S}^{3}\ci{K}:=\left\lbrace I\in\mathcal{S}:~I\subseteq K,~I\text{ intersects both }J\text{ and }I(J),~|I|\geq 2|I(J)|\right\rbrace,
\end{equation*}
\begin{equation*}
\mathcal{S}^{4}\ci{K}:=\lbrace I\in\mathcal{S}:~I\subseteq K,~I\text{ intersects both }J \text{ and }I(J),~|I|<2|I(J)|\rbrace.
\end{equation*}

(a) There holds
\begin{equation*}
\sum_{I\in\mathcal{S}\ci{K}^{m}}(\La w\Ra\ci{I})^{p}\La\sigma\Ra\ci{I}|I|\lesssim_{p}\frac{k3^{-k}w(K)}{1-\e},\qquad\sum_{I\in\mathcal{S}\ci{K}^{m}}(\La \sigma\Ra\ci{I})^{p'}\La w\Ra\ci{I}|I|\lesssim_{p}\frac{k\sigma(K)}{1-\e},\qquad\forall m=1,2,3.
\end{equation*}

(b) The testing conditions in \eqref{testing_conditions} hold in the case $L=K$.
\end{prop}
\begin{proof}
(a) By Lemma \ref{average_estimates} (b) we have
\begin{align*}
\sum_{I\in\mathcal{S}\ci{K}^{1}}(\La w\Ra\ci{I})^{p}\La\sigma\Ra\ci{I}|I|\lesssim_{p} 3^{-k}\sum_{I\in\mathcal{S}\ci{K}^{1}}\La w\Ra\ci{K}|I|\leq\frac{3^{-k}\La w\Ra\ci{K}|K|}{1-\e}=\frac{3^{-k}w(K)}{1-\e}.
\end{align*}
Let $S(J)$ be the triadic child of $K$ containing $I(J)$. All intervals in $\mathcal{S}^{2}\ci{K}$ are contained in $S(J)$, therefore by Lemma \ref{testing_intersect_only_I(J)_w} and Lemma \ref{average_estimates_triadic} we obtain
\begin{align*}
\sum_{I\in\mathcal{S}\ci{K}^{2}}(\La w\Ra\ci{I})^{p}\La\sigma\Ra\ci{I}|I|&\lesssim_{p}\frac{w(S(J))}{1-\e}=\frac{w(I(J))}{1-\e}\sim\frac{3^{-k}w(K)}{1-\e}.
\end{align*}

Moreover, by Lemma \ref{average_estimates} (c) and Lemma \ref{average_estimates_triadic} we have
\begin{align*}
\sum_{I\in\mathcal{S}\ci{K}^{3}}(\La w\Ra\ci{I})^{p}\La\sigma\Ra\ci{I}|I|&\lesssim_{p}\sum_{I\in\mathcal{S}_{k}^{3}}(\La w\Ra\ci{I(J)})^{p}\La\sigma\Ra\ci{I(J)}\frac{|I(J)|}{|I|}|I|\sim_{p}\La w\Ra\ci{K}\sum_{I\in\mathcal{S}_{k}^{3}}|I(J)|=3^{-k}w(K)(\#\cS\ci{K}^3).
\end{align*}
Notice that all elements of $\mathcal{S}^{3}\ci{K}$ contain the common endpoint of $J$ and $I(J)$, therefore $\mathcal{S}^3\ci{K}$ is linearly ordered with respect to containment. Denoting by $I_1$ the largest interval in $\cS^3\ci{K}$ and by $I\ci{N}$ the smallest one, where $N:=\#\cS^{3}\ci{K}$, we obtain
\begin{align*}
3^{k}|I(J)|=|K|\geq|I_1|\geq\frac{|I\ci{N}|}{\e^{N-1}}\geq\frac{2|I(J)|}{\e^{N-1}},
\end{align*}
therefore
\begin{equation*}
\#\cS\ci{K}^3=N\leq \frac{k\log 3-\log 2}{\log(1/\e)}+1\lesssim\frac{k}{1-\e},
\end{equation*}
concluding the proof of the first three estimates.

The other three estimates are proved similarly, recalling from Lemma \ref{average_estimates_triadic} that $\sigma(I(J))\sim_{p} \sigma(K)$ and $\La\sigma\Ra\ci{I(J)}\sim_{p} 3^{k}\La\sigma\Ra\ci{K}$.

(b) It is clear that
\begin{equation*}
\sum_{\substack{I\in\mathcal{S}\\I\subseteq K}}(\La w\Ra\ci{I})^{p}\La\sigma\Ra\ci{I}|I|=
\sum_{\substack{K'\in\mathbf{K}\\K'\subseteq K}}\sum_{m=1}^{4}\sum_{I\in\mathcal{S}^{m}\ci{K'}}(\La w\Ra\ci{I})^{p}\La\sigma\Ra\ci{I}|I|.
\end{equation*}
By (a) and Lemma \ref{packing_K} we have
\begin{align*}
\sum_{\substack{K'\in\mathbf{K}\\K'\subseteq K}}\sum_{m=1}^{3}\sum_{I\in\mathcal{S}^{m}\ci{K'}}(\La w\Ra\ci{I})^{p}\La\sigma\Ra\ci{I}|I|\lesssim_{p}
\sum_{\substack{K'\in\mathbf{K}\\K'\subseteq K}}\frac{k3^{-k}w(K')}{1-\e}\sim\frac{kw(K)}{1-\e}.
\end{align*}
Moreover, notice that all intervals in $\mathcal{S}^{4}\ci{K}$ are contained in a subinterval of $K$ of length (say) $5|I(J)|$, therefore by Lemma \ref{average_estimates} and Remark \ref{A_p} we obtain
\begin{align*}
\sum_{\substack{K'\in\mathbf{K}\\K'\subseteq K}}\sum_{I\in\mathcal{S}^{4}\ci{K'}}(\La w\Ra\ci{I})^{p}\La\sigma\Ra\ci{I}|I|&\lesssim_{p}
\sum_{\substack{K'\in\mathbf{K}\\K'\subseteq K}}\La w\Ra\ci{I((K')^m)}\sum_{I\in\mathcal{S}^{4}\ci{K'}}|I|\leq\frac{5}{1-\e}\sum_{\substack{K'\in\mathbf{K}\\K'\subseteq K}}\La w\Ra\ci{I((K')^m)}|I((K')^{m})|\\
&=\frac{5}{1-\e}\sum_{\substack{K'\in\mathbf{K}\\K'\subseteq K}}w(I((K')^m))=
\frac{5}{1-\e}w(K),
\end{align*}
concluding the proof of the first estimate. The second one is proved similarly.
\end{proof}

\begin{lm}\label{testing_contains_endpoint_K_w}
Let $K\in\mathbf{K}$, and let $L$ be any subinterval of $K$ sharing an endpoint with $K$. Then, the estimates in \eqref{testing_conditions} hold for $L$.
\end{lm}
\begin{proof}
We prove only the first estimate in \eqref{testing_conditions}, noting that the second one is proved similarly.

Set $J:=K^{m}$. We denote by $A\ci{J}$ the set of all triadic subintervals of $J$ of length $3^{1-k}|J|$, and we enumerate $A\ci{J}:=\lbrace K\ci{1,J},\ldots,K\ci{3^{k-1},J}\rbrace$, so that successive intervals in this enumeration are adjacent and $I(J),K\ci{1,J}$ are adjacent. Note that $A\ci{J}=\text{ch}\ci{\mathbf{K}}(K)$.

We denote by $A\ci{1,J}$ the family of all intervals in $A\ci{J}$ that intersect $L$, and by $A\ci{2,J}$ the family of all intervals in $A\ci{J}$ that are contained in $L$. We will say that $L$ is \textbf{good} for $K$ if either $A\ci{1,J}=A\ci{2,J}$ or $L$ contains at least one interval in $A\ci{J}\cup\lbrace I(J)\rbrace$.

Assume first that $L$ is good. By Lemma \ref{testing_condition_K_w} and disjointness of the  intervals in $A\ci{J}$ we deduce
\begin{align}
\label{good_1}
\sum_{K'\in A\ci{2,J}}\sum_{\substack{I\in\mathcal{S}\\~I\subseteq K'}}(\La w\Ra\ci{I})^{p}\La\sigma\Ra\ci{I}|I|\lesssim_{p}\sum_{K'\in A\ci{2,J}}\frac{kw(K')}{1-\e}\leq \frac{kw(L)}{1-\e}.
\end{align}
If $A\ci{1,J}=A\ci{2,J}$, then combining Lemma \ref{partial_testing_w} with  \eqref{good_1} we deduce the desired result.

Assume now that $A\ci{1,J}$ and $A\ci{2,J}$ do not coincide. Then, there is exactly one element $K'$ of $A\ci{1,J}$ not belonging to $A\ci{2,J}$, for $L$ is not contained in $J$. Since $L$ contains at least one interval in $A\ci{J}\cup\lbrace I(J)\rbrace$, by Lemma \ref{average_estimates_triadic} we obtain $w(K')\leq w(L)$ (and also $\sigma(K')\lesssim_{p} \sigma(L)$), and thus by Lemma \ref{testing_condition_K_w} we obtain
\begin{equation}
\label{good_2}
\sum_{\substack{I\in\mathcal{S}\\I\subseteq L\cap K'}}(\La w\Ra\ci{I})^{p}\La\sigma\Ra\ci{I}|I|\lesssim_{p}\frac{kw(K')}{1-\e}\lesssim_{p}\frac{kw(L)}{1-\e}.
\end{equation}
Combining Lemma \ref{partial_testing_w} with \eqref{good_1} and \eqref{good_2} we deduce the desired estimate.

Assume now that $L$ is not good. Then, it is clear that $A\ci{1,J}=\lbrace K\ci{3^{k-1},J}\rbrace$ and $A\ci{2,J}=\emptyset$, so in particular
\begin{equation}
\label{predecomp}
L\cap(J\cup I(J))=L\cap K\ci{3^{k-1},J}.
\end{equation}
Set
\begin{equation*}
K_1:=K,~J_1:=(K_1)^m,~K_2:=K\ci{3^{k-1},J_1},~J_2:=(K_2)^{m} ,
\end{equation*}
and
\begin{equation*}
L_1:=L,~L_2:=L_1\cap K_2.
\end{equation*}
Note that \eqref{predecomp} coupled with the fact that $L$ and $K_1$ share an endpoint implies that
\begin{equation}
\label{one step decomposition}
L=R(J_1)\cup L_2,
\end{equation}
where $R(J_1)$ denotes the triadic child of $K_1$ not containing $I(J_1)$. Also $|L_{2}|\leq |K_2|=3^{1-k}|R(J_1)|\leq 3^{1-k}|L_1|$.

Next, we repeat the above procedure for the interval $L_2\subseteq K_2$, noting that $L_2$ shares an endpoint with $K_2$. 

Continuing thus inductively, we construct (possibly finite) sequences of intervals $K_1,K_2,\ldots$, $L_1,L_2,\ldots$  and $J_1,J_2,\ldots$, such that
\begin{equation*}
K_1=K,\qquad J_{l}=(K_l)^m,\qquad K_{l+1}=K\ci{3^{k-1},J_l},\qquad\forall l=1,2,\ldots,
\end{equation*}
\begin{equation*}
L_1=L,\qquad L_{l+1}=L_{l}\cap K_{l+1}=L_{l}\cap(J_{l}\cup I(J_{l})),\qquad\forall l=1,2,\ldots,
\end{equation*}
\begin{equation*}
L_{l},K_{l}\text{ share an endpoint, for all }l=1,2,\ldots.
\end{equation*}
The construction terminates at some positive integer $N\geq 2$ if and only if the interval $L\ci{N}$ is good for $K\ci{N}$. Note that induction using \eqref{one step decomposition} shows that
\begin{equation*}
L=R(J_1)\cup\ldots R(J_{l-1})\cup L_{l},\qquad\forall l=2,3,\ldots.
\end{equation*}
In particular, $w,\sigma$ vanish in $L\setminus K_{l}$, $L_{l}=L\cap K_{l}$,  and $|L_{l}|\leq 3^{-l+1}|L|$, for all $l=1,2,3,\ldots$. We now distinguish two cases.

\textbf{Case 1.} The construction terminates at some positive integer $N\geq 2$. Then, the family of all intervals in $\cS$ contained in $L$ and on which $w,\sigma$ might not vanish is contained in the union of the following two families:
\begin{equation*}
\mathcal{S}\ci{L}^1:=\lbrace I\in\mathcal{S}:~I\subseteq L,~I\cap K\ci{N}\neq\emptyset,~I\nsubseteq K\ci{N}\rbrace,\qquad\mathcal{S}\ci{L}^2:=\lbrace I\in\mathcal{S}:~I\subseteq L\ci{N}\rbrace.
\end{equation*}
It is clear that for every $I\in\mathcal{S}\ci{L}^1$, the intervals $I\cap K\ci{N},L\ci{N},K\ci{N}$ share an endpoint, and $I\cap K\ci{N}\subseteq L\ci{N}$. Since $w,\sigma$ vanish on $L\setminus K\ci{N}$, it follows from Lemma \ref{comparison_contain_endpoint}, Remark \ref{A_p} and Lemma \ref{restricted_sparse_packing} that
\begin{align*}
\sum_{I\in\mathcal{S}^1\ci{L}}(\La w\Ra\ci{I})^{p}\La\sigma\Ra\ci{I}|I|&=\sum_{I\in\mathcal{S}^1\ci{L}}\left(\frac{|I\cap K\ci{N}|}{|I|}\right)^{p+1}(\La w\Ra\ci{I\cap K\ci{N}})^{p}\La\sigma\Ra\ci{I\cap K\ci{N}}|I|\\
&\lesssim_{p}\La w\Ra\ci{L\ci{N}}\sum_{I\in\mathcal{S}^1\ci{L}}\left(\frac{|I\cap K\ci{N}|}{|I|}\right)^{p+1}|I|\lesssim_{p}\frac{\La w\Ra\ci{L\ci{N}}|L\cap K\ci{N}|}{1-\e}=\frac{w(L)}{1-\e}.
\end{align*}
Moreover, since $L\ci{N}$ is good we have
\begin{equation*}
\sum_{I\in\mathcal{S}^2\ci{L}}(\La w\Ra\ci{I})^p\La\sigma\Ra\ci{I}|I|\lesssim_{p}\frac{kw(L\ci{N})}{1-\e}=\frac{kw(L)}{1-\e},
\end{equation*}
concluding the proof in this case.

\textbf{Case 2.} Assume that the sequence never terminates. Then $L$ is contained in $\bigcup_{l=1}^{\infty}R(J_l)$ up to a set of zero measure (in fact a singleton). It follows that $w,\sigma$ vanish a.e on $L$, and thus we have nothing to show.
\end{proof}

Now we are in a position to prove Proposition \ref{p: local_testing_condition} in full generality.

\begin{proof}[Proof (of Proposition \ref{p: local_testing_condition})]
We prove only the first estimate in \eqref{testing_conditions}, noting that the second one is proved similarly.

Let $K$ be the smallest interval in $\mathbf{K}$ containing $L$, and set $J:=K^m$. We denote by $\wt{A}\ci{J}$ the set of all triadic subintervals of $J$ of length $3^{1-k}|J|$ intersecting $L$. In view of Lemma \ref{partial_testing_w}, it suffices to proves that
\begin{equation*}
\sum_{K'\in\wt{A}\ci{J}}\sum_{\substack{I\in\mathcal{S}\\I\subseteq L\cap K'}}(\La w\Ra\ci{I})^{p}\La\sigma\Ra\ci{I}|I|\lesssim_{p}\frac{kw(L)}{1-\e}.
\end{equation*}
It is clear that $L\cap K'$ is a subinterval of $K'$ sharing an endpoint with $K'$, for all $K'\in \wt{A}\ci{J}$. Therefore, it follows from Lemma \ref{testing_contains_endpoint_K_w} and disjointness of the intervals in $\wt{A}\ci{J}$ that
\begin{align*}
\sum_{K'\in \wt{A}\ci{J}}\sum_{\substack{I\in\mathcal{S}\\I\subseteq L\cap K'}}(\La w\Ra\ci{I})^{p}\La\sigma\Ra\ci{I}|I|\lesssim_{p}\sum_{K'\in \wt{A}\ci{J}}\frac{kw(L\cap K')}{1-\e}\leq\frac{kw(L)}{1-\e},
\end{align*}
concluding the proof.
\end{proof}

\subsection{Verifying global testing conditions} \label{s: direct sum}

In this subsection we give details on the verification of the global testing conditions \eqref{global testing conditions}. For all positive integers $k>3000$, we denote by $w_{k},\sigma_{k}$ the weights of Section \ref{s: reguera-thiele} that were constructed for these $k,p$. Choose some
\begin{equation*}
r\in\left(\max\left(1,\frac{1}{p-1}\right),p'\right),
\end{equation*}
where we recall that $p'=\frac{p}{p-1}$. Set
\begin{equation*}
\wt{w}_{k}:=\frac{1}{k^r}w_{k},~k=4000,4001,4002,\ldots.
\end{equation*}
Following Reguera--Scurry \cite{reguera-scurry}, we consider the weights $\wt{w},\sigma$ on $\R$ given by
\begin{equation*}
\wt{w}(x):=\sum_{k=4000}^{\infty}\wt{w}_{k}(x-9^{k}),~\sigma(x):=\sum_{k=4000}^{\infty}\sigma_{k}(x-9^{k}),\qquad x\in\R.
\end{equation*}
Note that $\wt{w},\sigma$ vanish outside $\bigcup_{k=4000}^{\infty}I_{k}$, where
\begin{equation*}
I_{k}:=[9^{k},9^{k}+1),~\forall k=4000,4001,\ldots,
\end{equation*}
and that $\wt{w}(I_{k})=k^{-r},~\sigma(I_{k})\sim_{p}3^{-k}$, for all $k=4000,4001,\ldots$. 

\begin{prop}\label{prop: global}
Let $0<\e<1$, and let $\cS$ be any martingale $\e$-sparse family of intervals in $\R$. Let $L$ be any interval in $\R$. Then
\begin{equation*}
\sum_{\substack{I\in\mathcal{S}\\I\subseteq L}}(\La\wt{w}\Ra\ci{I})^{p}\La\sigma\Ra\ci{I}|I|\lesssim_{p}\frac{\wt{w}(L)}{1-\e},\qquad\sum_{\substack{I\in\mathcal{S}\\I\subseteq L}}(\La\sigma\Ra\ci{I})^{p'}\La \wt{w}\Ra\ci{I}|I|\lesssim_{p}\frac{\sigma(L)}{1-\e}.
\end{equation*}
\end{prop}
\begin{proof}
Clearly
\begin{equation*}
\sum_{\substack{I\in\mathcal{S}\\I\subseteq L}}(\La\wt{w}\Ra\ci{I})^{p}\La\sigma\Ra\ci{I}|I|=\sum_{i=1}^{2}\sum_{\substack{I\in\cS^{i}\\I\subseteq L}}(\La\wt{w}\Ra\ci{I})^{p}\La\sigma\Ra\ci{I}|I|,
\end{equation*}
where
\begin{equation*}
\cS^1:=\left\lbrace I\in\cS:~ I\subseteq L,~ |I|\leq 2, ~I\cap\left(\bigcup_{k=4000}^{\infty}I_{k}\right)\neq\emptyset\right\rbrace,~\cS^2:=\lbrace I\in\cS:~ I\subseteq L,~ |I|>2\rbrace.
\end{equation*}
Notice that for all $I\in\cS^1$ there is a unique $k\geq 4000$ such that $I\cap I_{k}\neq\emptyset$.

Let $k\geq 4000$. We have
\begin{equation*}
\sum_{\substack{I\in\mathcal{S}^1\\I\cap I_{k}\neq\emptyset}}(\La\wt{w}\Ra\ci{I})^{p}\La\sigma\Ra\ci{I}|I|=\sum_{\substack{I\in\mathcal{S}\\ I\subseteq L\cap I_{k}}}(\La\wt{w}\Ra\ci{I})^{p}\La\sigma\Ra\ci{I}|I|+\sum_{\substack{I\in\mathcal{S}^1\\
I\cap I_{k}\neq\emptyset,~I\nsubseteq I_{k}}}(\La\wt{w}\Ra\ci{I})^{p}\La\sigma\Ra\ci{I}|I|.
\end{equation*}
By the first estimate in  \eqref{rescaled local testing conditions} coupled with translation invariance we obtain
\begin{equation*}
\sum_{\substack{I\in\mathcal{S}\\I\subseteq L\cap I_{k}}}(\La\wt{w}\Ra\ci{I})^{p}\La\sigma\Ra\ci{I}|I|\lesssim_{p}\frac{\wt{w}(L\cap I_{k})}{1-\e}.
\end{equation*}
Moreover, if $I\in\cS^1$ is such that $I\cap I_{k}\neq\emptyset$ and $I\nsubseteq I_{k}$, then the intervals $I\cap I_{k},L\cap I_{k}$ and $I_{k}$ share an endpoint and $I\cap I_{k}\subseteq L\cap I_{k}$, therefore by Remark \ref{A_p}, Lemma \ref{comparison_contain_endpoint} and Lemma \ref{restricted_sparse_packing} we obtain
\begin{align*}
\sum_{\substack{I\in\mathcal{S}\\I\subseteq L\\
I\cap I_{k}\neq\emptyset,~I\nsubseteq I_{k}}}(\La\wt{w}\Ra\ci{I})^{p}\La\sigma\Ra\ci{I}|I|&=
\sum_{\substack{I\in\mathcal{S}\\I\subseteq L\\
I\cap I_{k}\neq\emptyset,~I\nsubseteq I_{k}}}\left(\frac{|I\cap I_{k}|}{|I|}\right)^{p+1}(\La\wt{w}\Ra\ci{I\cap I_k})^{p}\La\sigma\Ra\ci{I\cap I_k}|I|\\
&\lesssim_{p}\sum_{\substack{I\in\mathcal{S}\\I\subseteq L\\
I\cap I_{k}\neq\emptyset,~I\nsubseteq I_{k}}}\La \wt{w}\Ra\ci{I\cap I_{k}}\left(\frac{|I\cap I_{k}|}{|I|}\right)^{p+1}|I|\lesssim_{p}\La\wt{w}\Ra\ci{L\cap I_{k}}\frac{|L\cap I_{k}|}{1-\e}=\frac{\wt{w}(L\cap I_{k})}{1-\e}.
\end{align*}
It follows that
\begin{align*}
\sum_{\substack{I\in\mathcal{S}^1}}(\La\wt{w}\Ra\ci{I})^{p}\La\sigma\Ra\ci{I}|I|=\sum_{k=4000}^{\infty}\sum_{\substack{I\in\mathcal{S}^1\\I\cap I_{k}\neq\emptyset}}(\La\wt{w}\Ra\ci{I})^{p}\La\sigma\Ra\ci{I}|I|\lesssim_{p}\sum_{k=4000}^{\infty}\frac{\wt{w}(L\cap I_{k})}{1-\e}=\frac{\wt{w}(L)}{1-\e}.
\end{align*}
Let now $m$ be any positive integer. Set $\cS^{2}_{m}:=\lbrace I\in\cS^2:~2^{m}< |I|\leq 2^{m+1}\rbrace$. It is easy to see that for all $I\in\cS^2_{m}$, one has $\#\lbrace k\geq 4000:~I\cap I_{k}\neq\emptyset\rbrace\leq 2m$, thus $\La\wt{w}\Ra\ci{I},\La\sigma\Ra\ci{I}\lesssim_{p}m2^{-m}$. Moreover, since $2^{m}< |I|\leq2^{m+1}$, for all $I\in\cS^{2}_{m}$, similarly to the proof of Lemma \ref{partial_testing_w} (a) we obtain
\begin{equation*}
\sum_{I\in\cS^{2}_{m}}\wt{w}(I)\lesssim\frac{\wt{w}(L)}{1-\e}.
\end{equation*}
Thus, we have
\begin{align*}
\sum_{I\in\cS^2}(\La\wt{w}\Ra\ci{I})^{p}\La\sigma\Ra\ci{I}|I|=\sum_{m=1}^{\infty}\sum_{I\in\cS^{2}_{m}}(\La\wt{w}\Ra\ci{I})^{p}\La\sigma\Ra\ci{I}|I|\lesssim_{p}
\sum_{m=1}^{\infty}m^{p}2^{-mp}\sum_{I\in\cS^{2}_{m}}\wt{w}(I)\lesssim_{p}\frac{\wt{w}(L)}{1-\e},
\end{align*}
concluding the proof of the first estimate. The second one is proved similarly.
\end{proof}

\section{Investigating separated bump conditions} \label{s: entropy bumps}

Throughout this section we assume $p=2$. By interval we mean a subset of $\R$ of the form $[a,b)$, where $a,b\in\R$, $a<b$.

\subsection{Lorentz spaces} \label{s: lorentz norms} We record here well-known facts about Lorentz spaces.

A function $\phi:[0,\infty)\rightarrow[0,\infty)$ is said to be \emph{quasiconcave} if
\begin{itemize}
\item[(a)] $\phi$ is increasing

\item[(b)] $\phi(s)=0$ is and only if $s=0$, for all $s\in[0,\infty)$

\item[(c)] the function $t\mapsto \phi(t)/t$ is decreasing on $(0,\infty)$.
\end{itemize}
As noted in \cite[Chapter 2, Corollary 5.3]{interpolation}, combining properties (a) and (c) we deduce that every quasiconcave function is continuous at each point of $(0,\infty)$. It is also proved in \cite[Chapter 2, Proposition 5.10]{interpolation} that every quasiconcave function $\phi:[0,\infty)\rightarrow[0,\infty)$ admits a least concave majorant $\wt{\phi}:[0,\infty)\rightarrow[0,\infty)$ satisfying
\begin{equation}
\label{least concave majorant}
\frac{1}{2}\wt{\phi}(x)\leq\phi(x)\leq\wt{\phi}(x),\qquad\forall x\in(0,\infty).
\end{equation}
It is not hard to see that $\wt{\phi}$ is also quasiconcave.

Fix now a quasiconcave function $\phi:[0,\infty)\rightarrow[0,\infty)$ and a non-atomic probability space $(X,\mu)$. We define for all measurable functions $f$ on $X$ the \textit{distribution function} $N_{f}$ of $f$ by
\begin{equation*}
N_{f}(t):=\mu(\lbrace|f|>t\rbrace),~0\leq t<\infty,
\end{equation*}
and the \textit{decreasing rearrangement} $f^{\ast}$ of $f$ by
\begin{equation*}
f^{\ast}(t):=\inf\lbrace s\in[0,\infty):N_{f}(s)\leq t\rbrace,~0\leq t\leq 1.
\end{equation*}
Following \cite[Chapter 2, Definition 5.12]{interpolation}, we define the \textit{Lorentz space} $\Lambda_{\phi}(X,\mu)$ as the space of all measurable functions $f$ on $X$ for which
\begin{equation}
\label{Lorentz norm}
\Vert f\Vert\ci{\Lambda_{\phi}(X,\mu)}:=\int_{[0,1]}f^{\ast}(s)d\phi(s)<\infty,
\end{equation}
where the integral is to be understood in the Lebesgue-Stieltjes sense. It is proved in \cite[Chapter 2., Theorem 5.13]{interpolation} that if $\phi$ is concave, then $\Vert\fdot\Vert\ci{\Lambda_{\phi}(X,\mu)}$ defines a norm on the space $\Lambda_{\phi}(X,\mu)$.

It is noted in \cite[Section 3, equation (3.3)]{entropy}  that the change of variables $s=N_{f}(t)$ and integration by parts show that that one can rewrite $\Vert f\Vert\ci{\Lambda_{\phi}(X,\mu)}$ as
\begin{equation}
\label{computation Lorentz}
\Vert f\Vert\ci{\Lambda_{\phi}(X,\mu)}=\int_{[0,\infty)}\phi(N_{f}(t))dt.
\end{equation}
In fact, using just the facts that $\phi(0)=0$ and that $\phi$ is left-continuous one can prove \eqref{computation Lorentz} by direct computation in the case that $N_{f}$ is a step function, and then the general case follows by approximating $N_{f}$ by an increasing sequence of right-continuous decreasing step functions. This way of writing Lorentz quasinorms is more useful for explicit computations.

In the special case that $(X,\mu)$ coincides with $(I,dx/|I|)$, where $I$ is an interval in $\R$ and $dx/|I|$ normalized Lebesgue measure on $I$, we denote $\Lambda_{\phi}(X,\mu)$ by $\Lambda_{\phi}(I)$.

\subsection{$L\log L$ space}\label{s: Llog L} We record here well-known facts about the space $L\log L$.

Let $(X,\mu)$ be a non-atomic probability space. Consider the continuous, strictly increasing, and concave function $\phi_0:[0,1]\rightarrow[0,\infty)$ given by
\begin{equation*}
\phi_0(s):=s(1-\log s),~0<s\leq 1,\qquad \phi_0(0)=0.
\end{equation*}
Consider also the Young function $\Phi_0(t):=t(\log t)^{+}$, $0\leq t<\infty$. The Orlicz space $L^{\Phi_{0}}(X,\mu)$ is denoted by $L\log L(X,\mu)$. It is proved in \cite[Chapter 4, Section 8]{interpolation} that the spaces $\Lambda_{\phi_0}(X,\mu)$ and $L\log L(X,\mu)$ coincide, and that the Lorentz norm $\Vert\fdot\Vert\ci{\Lambda_{\phi_0}(X,\mu)}$ and the Orlicz norm $\Vert\fdot\Vert\ci{L\log L(X,\mu)}$ are equivalent. Note that in the special case that $(X,\mu)$ coincides with $(I,dx/|I|)$, where $I$ is an interval in $\R$ and $dx/|I|$ normalized Lebesgue measure on $I$, the norms $\Vert\fdot\Vert\ci{L\log L(I)}$ and $\Vert\fdot\Vert\ci{\Lambda_{\phi_0}(I)}$ are equivalent with constants not depending on $I$, due to translation and rescaling invariance. Moreover, if $M\ci{I}$ denotes the Hardy--Littlewood maximal function adapted to $I$, that is
\begin{equation*}
M\ci{I}f:=\sup_{J}\La|f|\Ra\ci{J}1\ci{J},
\end{equation*}
where supremum is taken over all subintervals $J$ of $I$, then there holds
\begin{equation*}
\Vert f\Vert\ci{L\log L(I)}\sim\La M\ci{I}f\Ra\ci{I}.
\end{equation*}
In what follows, we denote $\Vert\fdot\Vert\ci{\Lambda_{\phi_0}(I)}$ by $\Vert\fdot\Vert\ci{I}^{\ast}$, for all intervals $I$.

\subsection{Comparison principles between Orlicz bumps and Lorentz bumps}\label{s: comparison principles}

We record here three principles of comparison between Orlicz bumps and Lorentz bumps, which allow us to reduce estimates for Orlicz bumps to estimates for Lorentz bumps, thus greatly simplifying computations.

\subsubsection{Estimating Orlicz bumps from below}

Let $\Phi$ be a Young function with $\int_{c}^{\infty}\frac{1}{\Phi(t)}dt<\infty$ for some $c>0$. Treil and Volberg prove in \cite[Lemma 3.4]{entropy} that $\Phi$ must then satisfy $\Phi(t)\geq a\cdot t\log t$ for all sufficiently large $t$, for some $a>0$. In fact, if we just assume tha $\Phi:[0,\infty)\rightarrow[0,\infty)$ is a function such that the function $t\mapsto\Phi(t)/t$ is increasing in $(0,\infty)$ and $\int_{c}^{\infty}\frac{1}{\Phi(t)}dt<\infty$ for some $c>1$, then adapting the proof of \cite[Lemma 3.4]{entropy} we have for all $t\geq c^2$
\begin{equation}
\label{domination of entropy}
A:=\int_{c}^{\infty}\frac{1}{\Phi(s)}ds\geq\int_{c}^{t}\frac{s}{s\Phi(s)}ds\geq\frac{t}{\Phi(t)}\int_{c}^{t}\frac{1}{s}ds=\frac{t(\log t-\log c)}{\Phi(t)}\geq\frac{t\log t}{2\Phi(t)},
\end{equation}
and therefore $\Phi(t)\geq\frac{1}{2A}t\log t$ (a slightly more careful variant of this argument shows that actually $\lim_{t\rightarrow\infty}\frac{\Phi(t)}{t\log t}=\infty$). It then follows from the facts about the $L\log L$ space recorded in Subsection \ref{s: Llog L} that
\begin{equation*}
\Vert f\Vert\ci{I}^{\ast}\lesssim\ci{\Phi}\Vert f\Vert\ci{L^{\Phi}(I)},
\end{equation*}
for all measurable functions $f\geq0$ on $I$, for all intervals $I$.

Treil--Volberg \cite{entropy} have shown that if the Young function $\Phi$ possesses mild additional regularity, then more can be said. Namely, assume that
\begin{itemize}
\item $\Phi$ is doubling, i.e.  there exists a positive constant $C$ such that $\Phi(2t)\leq C\Phi(t)$, for all $t>0$

\item the function $t\mapsto \Phi(t)/(t\log t)$ is increasing for sufficiently large $t$.
\end{itemize}
As Treil--Volberg point out in \cite{entropy},  these additional regularity assumptions are satisfied when $\Phi$ is a standard logarithmic bound of the form (for sufficiently large $t$)
\begin{equation*}
\Phi(t)=t(\log t)(\log^{(2)}t)\fdot\ldots\fdot(\log^{(n-1)} t)(\log^{(n)}t)^{1+\e},
\end{equation*}
for some $\e>0$ and some positive integer $n$, where $\log^{(k)} t$ denotes $k$-fold composition of $\log$ with itself.

The following comparison principle between Orlicz bumps and ``penalized'' entropy bumps is established by Treil and Volberg in \cite{entropy}. Recall that $\phi_0(s):=s(1-\log s)$, for all $s\in(0,1]$ and $\phi_0(0)=0$.

\begin{lm}[Treil--Volberg \cite{entropy}]\label{l: entropy comparison}
Let $(X,\mu)$ be a nonatomic probability space. Assume that the Young function $\Phi$ satisfies the above integrability and regularity conditions. Then, there exists a function $\alpha:[1,\infty)\rightarrow(0,\infty)$ (depending only on $\Phi$), such that the function $t\mapsto t\alpha(t)$ is increasing and
\begin{equation*}
\int_{1}^{\infty}\frac{1}{t\alpha(t)}dt<\infty,
\end{equation*}
satisfying
\begin{equation*}
\alpha\left(\frac{\Vert f\Vert\ci{\Lambda_{\phi_0}(X,\mu)}}{\Vert f\Vert\ci{L^1(\mu)}}\right)\Vert f\Vert\ci{\Lambda_{\phi_0}(X,\mu)}\leq\Vert f\Vert\ci{L^{\Phi}(X,\mu)},
\end{equation*}
for all measurable functions $f\geq0$ on $X$ that are positive on a set of positive measure.
\end{lm}

We refer to \cite{entropy} for the proof of Lemma \ref{l: entropy comparison}.

\subsubsection{Estimating Orlicz bumps from above}

 Let $(X,\mu)$ be a non-atomic probability space. Let $(\mathcal{X},\Vert\fdot\Vert\ci{\mathcal{X}})$ be a rearrangement invariant Banach function space on $(X,\mu)$.  Then, the fundamental function of $\mathcal{X}$ is defined to be the function $\phi\ci{\mathcal{X}}:[0,\infty)\rightarrow[0,\infty)$ given by
\begin{equation*}
\phi\ci{\mathcal{X}}(t):=\Vert 1\ci{E_{t}}\Vert\ci{\mathcal{X}},
\end{equation*}
where $E_t$ is any measurable subset of $X$ such that $\mu(E_t)=t$, for all $t\in[0,1]$, and $\phi\ci{\mathcal{X}}(t)=\phi\ci{\mathcal{X}}(1)$, for all $t\in(1,\infty)$. It is proved in \cite[Chapter 2, Corollary 5.3]{interpolation} that $\phi\ci{\mathcal{X}}$ is a quasiconcave function. Moreover, combining \eqref{least concave majorant}, \cite[Chapter 2, Proposition 5.11]{interpolation} and \cite[Chapter 2, Theorem 5.13]{interpolation} with expression \eqref{computation Lorentz} for Lorentz quasinorms we deduce that
\begin{equation}
\label{RIBFS vs Lorentz}
\Vert f\Vert\ci{\mathcal{X}}\leq2\Vert f\Vert\ci{\Lambda_{\phi\ci{\mathcal{X}}}(X,\mu)},
\end{equation}
for all measurable functions $f\geq 0$ on $X$.

Let now $\Phi$ be a Young function. It is proved in \cite[Chapter 4, Section 8]{interpolation} that the Orlicz space $L^{\Phi}(X,\mu)$ equipped with the Luxemburg norm $\Vert\fdot\Vert\ci{L^{\Phi}(X,\mu)}$ is a rearrangement-invariant Banach function space. Direct computation shows that the fundamental function $\phi$ of $L^{\Phi}(X,\mu)$ is given by
\begin{equation*}
\phi(s)=\frac{1}{\Phi^{-1}(1/s)},~0<s\leq1,\qquad\phi(0)=0,
\end{equation*}
where $\Phi^{-1}$ is the right-continuous inverse of $\Phi$ given by
\begin{equation*}
\Phi^{-1}(t):=\sup\lbrace s\in[0,\infty):~\Phi(s)\leq t\rbrace,\qquad0\leq t<\infty.
\end{equation*}
Moreover, by \eqref{RIBFS vs Lorentz} we obtain
\begin{equation}
\label{Orlicz dominated by Lorentz}
\Vert f\Vert\ci{L^{\Phi}(X,\mu)}\leq2\Vert f\Vert\ci{\Lambda_{\phi}(X,\mu)},
\end{equation}
for all measurable functions $f\geq 0$ on $X$. It is not hard to see that $\int_{c}^{\infty}\frac{1}{\Phi(t)}dt<\infty$ for some $c>0$ if and only if $\int_{0}^{1}\frac{1}{\phi(t)}dt<\infty$.

Another way to estimate Orlicz bumps from above using Lorentz bumps is provided by Nazarov, Reznikov, Treil and Volberg in the journal version of \cite{nazarov--reznikov--treil--volberg}.

\begin{lm}[Nazarov, Reznikov, Treil, Volberg \cite{nazarov--reznikov--treil--volberg}]\label{l: lorentz comparison}
Let $\Phi$ be a Young function such that
\begin{itemize}

\item $\int_{c}^{\infty}\frac{1}{\Phi(t)}dt<\infty$, for some $c>0$

\item $\Phi$ is doubling

\item one can write $\Phi(t)=t\rho(t)$, where $\rho:[0,\infty)\rightarrow[1,\infty)$ is a continuously differentiable and strictly increasing function such that $\lim_{t\rightarrow\infty}\rho(t)=\infty$, $t\mapsto t\rho'(t)/\rho(t)$ is decreasing for sufficiently large $t$, and
\begin{equation*}
\lim_{t\rightarrow\infty}\frac{t\rho'(t)}{\rho(t)}\log(\rho(t))=0.
\end{equation*}
\end{itemize}
Then, there exists a quasiconcave function $\psi$ on $[0,1]$ with $\int^{1}_{0}\frac{1}{\psi(s)}ds<\infty$ (depending only on $\Phi$), such that
\begin{equation*}
\Vert f\Vert\ci{L^{\Phi}(X,\mu)}\sim\ci{\Phi}\Vert f\Vert\ci{\Lambda_{\psi}(X,\mu)},
\end{equation*}
for every measurable function $f\geq 0$ on $I$. In particular, one can take (for sufficiently small $s\in(0,1)$) $\psi(s):=s\Psi(s)$, where $\Psi(s)$ is given implicitly by
\begin{equation*}
\Psi(s):=\Phi'(t)\qquad\text{ for }s=\frac{1}{\Phi(t)\Phi'(t)}.
\end{equation*}
\end{lm}

We refer to the journal version of \cite{nazarov--reznikov--treil--volberg} for the proof of Lemma \ref{l: lorentz comparison}.

\subsection{Blow-up of separated bump conditions.}\label{s: entropy-Orlicz comparison}

Let $\Phi$ be any Young function such that $\int_{c}^{\infty}\frac{1}{\Phi(t)}dt<\infty$ for some $c>0$. Recall the weights $\wt{w},\sigma$ of Subsection \ref{s: main sparse estimates}. We will show that
\begin{equation*}
\sup_{I}\Vert \wt{w}\Vert\ci{L^{\Phi}(I)}\La\sigma\Ra\ci{I}=\infty,
\end{equation*}
where supremum ranges over all intervals in $\R$. Recalling the definitions of the weights $\wt{w},\sigma$ from Subsection \ref{s: main sparse estimates}, in view of translation invariance it suffices to prove that
\begin{equation*}
\lim_{k\rightarrow\infty}k^{-r}\sup_{I}\Vert w_k\Vert\ci{L^{\Phi}(I)}\La\sigma_k\Ra\ci{I}=\infty,
\end{equation*}
where the supremum inside the limit ranges over all subintervals $I$ of $[0,1)$, $r\in(1,2)$, and for any integer $k>3000$ we denote by $w_k,\sigma_k$ the weights of Section \ref{s: reguera-thiele} corresponding to this $k$ and $p=2$.

In view of the comparison principles in Subsection \ref{s: comparison principles}, it suffices to prove that one can find subintervals $R_k,~k>4000$ of $[0,1)$ such that
\begin{equation*}
\lim_{k\rightarrow\infty}k^{-r}\Vert w_k\Vert^{\ast}\ci{R_k}\La\sigma_k\Ra\ci{R_k}=\infty.
\end{equation*}
Clearly, it suffices to show for all $k>4000$, one can find a subinterval $R_k$ of $[0,1)$ such that
\begin{equation*}
\Vert w_{k}\Vert\ci{R_k}^{\ast}\gtrsim 3^{k}\La w_{k}\Ra\ci{R_k},\qquad\La w_{k}\Ra\ci{R_k}\La\sigma_{k}\Ra\ci{R_k}\gtrsim 1.
\end{equation*}
In Subsection \ref{s: lorentz norm computations} we prove the following estimate.

\begin{lm}\label{l: entropy computation}
There holds
\begin{equation*}
\Vert w_{k}\Vert\ci{K}^{\ast}\sim3^{k}\La w_{k}\Ra\ci{K},~\forall K\in\mathbf{K}.
\end{equation*}
\end{lm}

Assume now Lemma \ref{l: entropy computation}. Fix $k>4000$. Choose any $K\in\mathbf{K}$, and set $J:=K^{m}$. Let $K'$ be the unique triadic subinterval of $J$ of length $3^{1-k}|J|=|I(J)|$ that is adjacent to $I(J)$. Consider the non-triadic interval
\begin{equation*}
R_k=R:=I(J)\cup K'.
\end{equation*}
Note that $|R|=2|K'|$. Recall that $\La w_k\Ra\ci{I(J)}=\La w_k\Ra\ci{K'}$ and $\La \sigma_k\Ra\ci{I(J)}\gtrsim \La \sigma_k\Ra\ci{K'}$, therefore $\La w_k\Ra\ci{R}=\La w_k\Ra\ci{I(J)}$ and $\La \sigma_k\Ra\ci{R}\sim \La\sigma_k\Ra\ci{I(J)}$. Thus $\La w_{k}\Ra\ci{R}\La\sigma_{k}\Ra\ci{R}\sim 1$, and it also follows from expression \eqref{computation Lorentz} for Lorentz norms and Lemma \ref{l: entropy computation} that
\begin{align*}
\Vert w_k\Vert\ci{R}^{\ast}&\geq\frac{1}{2}\Vert w_k\Vert\ci{K'}^{\ast}\gtrsim3^{k}\La w_{k}\Ra\ci{K'}= 3^{k}\La w_{k}\Ra\ci{R}.
\end{align*}

\subsection{An improvement for triadic intervals}\label{s: Lorentz-Orlicz comparison}

Consider the Young function $\Phi$ given by
\begin{equation*}
\Phi(t):=t\log(e+t)(\log(\log(e^e+t)))^r,\qquad0\leq t<\infty.
\end{equation*}
Clearly $\int_{1}^{\infty}\frac{1}{\Phi(t)}dt<\infty$. Our goal here is to prove that 
\begin{equation*}
\sup_{I\in\mathcal{T}}\Vert\wt{w}\Vert\ci{L^{\Phi}(I)}\La\sigma\Ra\ci{I}<\infty,\qquad \sup_{I\in\mathcal{T}}\Vert\sigma\Vert\ci{L^{\Phi}(I)}\La\wt{w}\Ra\ci{I}<\infty,
\end{equation*}
where $\mathcal{T}$ is the family of all triadic intervals in $\R$. Working similarly to Subsection \ref{s: direct sum} we see that it suffices to prove that
\begin{equation*}
\sup_{I\in\mathcal{T}([0,1))}\Vert\wt{w}_{k}\Vert\ci{L^{\Phi}(I)}\La\sigma_{k}\Ra\ci{I}\lesssim_{r} 1,\qquad \sup_{I\in\mathcal{T}([0,1))}\Vert \sigma_{k}\Vert\ci{L^{\Phi}(I)}\La \wt{w}_{k}\Ra\ci{I}\lesssim_{r}1,
\end{equation*}
where $\mathcal{T}([0,1))$ is the family of all triadic subintervals of $[0,1)$. Therefore, it suffices to prove that
\begin{equation}
\label{goal Lorentz bump}
\sup_{I\in\mathcal{T}([0,1))}\Vert w_{k}\Vert\ci{L^{\Phi}(I)}\La\sigma_{k}\Ra\ci{I}\lesssim_{r} k^r,\qquad \sup_{I\in\mathcal{T}([0,1))}\Vert \sigma_{k}\Vert\ci{L^{\Phi}(I)}\La w_{k}\Ra\ci{I}\lesssim_{r}1.
\end{equation}

Clearly $\Phi$ is strictly increasing on $[0,\infty)$. Let $\Phi^{-1}$ be the inverse of $\Phi$. Consider the function $\phi:[0,1]\rightarrow[0,\infty)$ given by
\begin{equation*}
\phi(s):=\frac{1}{\Phi^{-1}(1/s)},\qquad0<s\leq1,\qquad\phi(0)=0.
\end{equation*}
Consider also the function $\psi:[0,1]\rightarrow[0,\infty)$ given by
\begin{equation*}
\psi(s):=s(12-\log s)(\log(12-\log s))^r,\qquad0<s\leq1,\qquad\psi(0)=0.
\end{equation*}
Note that $\psi$ is continuous, strictly increasing and strictly convave. It is not hard to see, and we include a proof in the appendix, that
\begin{equation*}
\psi(s)\sim_{r} \phi(s),\qquad 0\leq s\leq1.
\end{equation*}
Combining this observation with \eqref{Orlicz dominated by Lorentz} we deduce that for all intervals $I$ we have
\begin{equation*}
\Vert f\Vert\ci{L^{\Phi}(I)}\lesssim_{r}\Vert f\Vert\ci{\Lambda_{\psi}(I)},
\end{equation*}
for all measurable functions $f\geq0$ on $I$. Therefore, \eqref{goal Lorentz bump} will follow once we show that
\begin{equation*}
\sup_{I\in\mathcal{T}([0,1))}\Vert w_{k}\Vert\ci{\Lambda_{\psi}(I)}\La \sigma_{k}\Ra\ci{I}\lesssim k^r,\qquad \sup_{I\in\mathcal{T}([0,1))}\Vert \sigma_{k}\Vert\ci{\Lambda_{\psi}(I)}\La  w_{k}\Ra\ci{I}\lesssim 1.
\end{equation*}
Recalling the expression \eqref{computation Lorentz} for Lorentz norms, a computation as in the proof of Lemma \ref{average_estimates_triadic} shows that it suffices to prove that
\begin{align*}
\sup_{K\in\mathbf{K}}\Vert w_k\Vert\ci{\Lambda_{\psi}(K)}\La\sigma_k\Ra\ci{K}\lesssim k^{r},\qquad\sup_{K\in\mathbf{K}}\Vert \sigma_{k}\Vert\ci{\Lambda_{\psi}(K)}\La  w_{k}\Ra\ci{K}\lesssim 1.
\end{align*}
Recalling from Lemma \ref{average_estimates_triadic} that $\La \sigma_{k}\Ra\ci{K}\sim 3^{-k}\La w_{k}\Ra\ci{K}^{-1}$, for all $K\in\mathbf{K}$, it suffices to prove that
\begin{equation}
\label{lorentz estimate}
\Vert w_{k}\Vert\ci{\Lambda_\psi(K)}\lesssim 3^{k}k^{r}\La w_{k}\Ra\ci{K},\qquad \Vert \sigma_{k}\Vert\ci{\Lambda_\psi(K)}\lesssim \La w_{k}\Ra\ci{K}^{-1},\qquad\forall K\in\mathbf{K}.
\end{equation}
The proof of \eqref{lorentz estimate} is given in Subsection \ref{s: lorentz norm computations} below.

\subsection{Computing Lorentz norms} \label{s: lorentz norm computations} In this subsection we estimate the Lorents norms $\Vert w_{k}\Vert\ci{\Lambda_{\phi}(K)},\Vert \sigma_{k}\Vert\ci{\Lambda_{\phi}(K)}$, $\phi\in\lbrace\phi_0,\psi\rbrace$, for any $K\in\mathbf{K}$, using expression \eqref{computation Lorentz} for Lorentz norms. In order to simplify the notation, we denote $w_k,\sigma_k$ by just $w,\sigma$ respectively.

In what follows, we fix $K\in\mathbf{K}$. Let $\phi:[0,1]\rightarrow[0,\infty)$ be a continuous concave increasing function with $\phi(0)=0$ and $\phi(s)>0$, for all $s\in(0,1]$. Let $i$ be the unique nonnegative integer such that $K\in\mathbf{K}_{i}$, so $|K|=3^{-ik}$. Note that $w|\ci{K},\sigma|\ci{K}$ vanish outside
\begin{equation*}
\bigcup_{m=i+1}^{\infty}\bigcup_{\substack{J\in\mathbf{J}_{m}\\J\subseteq K}}I(J).
\end{equation*}
For all $m\geq i+1$, we have $\#\mathbf{J}_{m}=3^{(m-i-1)(k-1)}$, and for all $J\in\mathbf{J}_{m}$, we have $|I(J)|=3^{-mk}$, $w|\ci{I(J)}\equiv\left(\frac{3^{k}}{3^{k-1}+1}\right)^{m}$ and $\sigma|\ci{I(J)}\equiv\left(\frac{3^{k}}{3^{k-1}+1}\right)^{-m}$. We also note that
\begin{align*}
\sum_{m=l+1}^{\infty}3^{(m-i-1)(k-1)}\fdot 3^{-mk}&=\frac{3}{2}\fdot3^{-(i+1)k}\fdot3^{i-l},\qquad\forall l=i,i+1,i+2,\ldots.
\end{align*}
It follows that
\begin{equation*}
N^{K}_{w}(t)=
\begin{cases}
\frac{3}{2}\fdot3^{-k},\text{ if }t<\left(\frac{3^{k}}{3^{k-1}+1}\right)^{i+1}\\\\
\frac{3}{2}\fdot 3^{-k}\fdot 3^{i-l},\text{ if }\left(\frac{3^{k}}{3^{k-1}+1}\right)^{l}\leq t<\left(\frac{3^{k}}{3^{k-1}+1}\right)^{l+1},\text{ for some }l\geq i+1
\end{cases}
,~0<t<\infty,
\end{equation*}
therefore expression \eqref{computation Lorentz} for Lorentz norms yields
\begin{align}
\label{general expression}
\Vert w\Vert\ci{\Lambda_{\phi}(K)}&=\left(\frac{3^{k}}{3^{k-1}+1}\right)^{i+1}\phi\left(\frac{3}{2}\cdot3^{-k}\right)+\sum_{l=i+1}^{\infty}\left[\left(\frac{3^{k}}{3^{k-1}+1}\right)^{l+1}-\left(\frac{3^{k}}{3^{k-1}+1}\right)^{l}\right]\phi\left(\frac{3}{2}\cdot3^{-k}\cdot3^{i-l}\right)\\
\nonumber&\sim\left(\frac{3^{k}}{3^{k-1}+1}\right)^{i+1}\sum_{l=0}^{\infty}\left(\frac{3^{k}}{3^{k-1}+1}\right)^{l}\phi\left(\frac{3}{2}\cdot3^{-k}\cdot3^{-l}\right)\\
\nonumber&\sim\La w\Ra\ci{K}\sum_{l=0}^{\infty}\left(\frac{3^{k}}{3^{k-1}+1}\right)^{l}\phi\left(3^{-k-l}\right),
\end{align}
since $\phi(s)\leq\phi(3s/2)\leq3\phi(s)/2$, for all $s\in[0,2/3]$. A similar computation shows that
\begin{equation*}
\Vert\sigma\Vert\ci{\Lambda_{\phi}(K)}\lesssim\phi(3^{-k})\left(\frac{3^{k}}{3^{k-1}+1}\right)^{-(i+1)}\lesssim\phi(1)\La w\Ra\ci{K}^{-1}.
\end{equation*}

\subsubsection{Estimating entropy bumps} Here we specialize to the case
\begin{equation*}
\phi(s)=\phi_0(s)=s(1-\log s),~0<s\leq1.
\end{equation*}
We have
\begin{align*}
&\sum_{l=0}^{\infty}\left(\frac{3^{k}}{3^{k-1}+1}\right)^{l}\phi_0(3^{-k-l})=3^{-k}\sum_{l=0}^{\infty}\left(\frac{3^{k-1}}{3^{k-1}+1}\right)^{l}(1+(l+k)\log 3)\\
&\sim3^{-k}k\sum_{l=0}^{\infty}\left(\frac{3^{k-1}}{3^{k-1}+1}\right)^{l}+3^{-k}\sum_{l=1}^{\infty}l\left(\frac{3^{k-1}}{3^{k-1}+1}\right)^{l}\sim3^{-k}k3^{k}+3^{-k}3^{2k}\sim 3^{k}.
\end{align*}
It follows from \eqref{general expression} that
\begin{align*}
\Vert w\Vert\ci{K}^{\ast}=\Vert w\Vert\ci{\Lambda_{\phi_0}(K)}\sim3^{k}\La w\Ra\ci{K}.
\end{align*}

\subsubsection{Estimating stronger Lorentz bumps} Here we specialize to the case
\begin{equation*}
\phi(s)=\psi(s)=s(12-\log s)(\log(12-\log s))^r,~0<s\leq1.
\end{equation*}
Recall that $1<r<2$ and $k>3000$. We have
\begin{align*}
\sum_{l=0}^{\infty}\left(\frac{3^{k}}{3^{k-1}+1}\right)^{l}\psi(3^{-k-l})&\sim 3^{-k}\sum_{l=2}^{\infty}\left(\frac{3^{k-1}}{3^{k-1}+1}\right)^{l}(k+l)((\log k)^r+(\log l)^r)\\
&=3^{-k}k(\log k)^{r}\sum_{l=2}^{\infty}\left(\frac{3^{k-1}}{3^{k-1}+1}\right)^{l}+
k3^{-k}\sum_{l=2}^{\infty}(\log l)^r\left(\frac{3^{k-1}}{3^{k-1}+1}\right)^{l}+\\
&+
3^{-k}(\log k)^{r}\sum_{l=1}^{\infty}l\left(\frac{3^{k-1}}{3^{k-1}+1}\right)^{l}+3^{-k}\sum_{l=2}^{\infty}l(\log l)^r\left(\frac{3^{k-1}}{3^{k-1}+1}\right)^{l}.
\end{align*}
We show in the appendix that
\begin{equation*}
\sum_{l=2}^{\infty}(\log l)^{r}x^{l}\lesssim\frac{(-\log(1-x))^{r}}{1-x},\qquad \sum_{l=2}^{\infty}l(\log l)^{r}x^{l}\lesssim\frac{(-\log(1-x))^{r}}{(1-x)^2},\qquad 0<x<1.
\end{equation*}
It follows that
\begin{align*}
\sum_{l=0}^{\infty}\left(\frac{3^{k}}{3^{k-1}+1}\right)^{l}\psi(3^{-k-l})&\lesssim 3^{-k}k(\log k)^{r}3^{k}+k3^{-k}k^r3^{k}+3^{-k}(\log k)^{r}3^{2k}+3^{-k}k^r3^{2k}\\
&\sim 3^{k}k^r,
\end{align*}
concluding the proof.

\section{Appendix}

\subsection{From martingale sparse families to general sparse families} Here we explain how estimates for sparse $p$-functions over general sparse families can be reduced to estimates for sparse $p$-functions over martingale sparse families. We follow \cite[Subsection 2.1]{convex_body}.

Let $1\leq p<\infty$ and $0<\eta<1$. Let $\cS$ be an $\eta$-sparse family of cubes in $\R^d$. We first recall the well-known ``three lattices trick'', see for instance \cite[Theorem 3.1]{intuitive}: there exist dyadic lattices $\cD^{j},~j=1,\ldots,3^d$ of cubes in $\R^d$, such that for all cubes $Q$ in $\R^d$ there exist $j\in\lbrace 1,\ldots,3^{d}\rbrace$ and $R\in\cD^j$ with $Q\subseteq R$ and $|R|\leq 6^{d}|Q|$. Using this, it is easy to see that there exist $\frac{6^{d}+\eta-1}{6^{d}}$-sparse families $\cS^{j},~j=1,\ldots,3^{d}$, such that $\cS^{j}\subseteq\cD^{j}$, for all $j=1,\ldots,3^{d}$, and
\begin{equation*}
\mathcal{A}\ci{\cS,p}f\leq 6^{d}\sum_{j=1}^{3^{d}}\mathcal{A}\ci{\cS^j,p}f,\qquad\forall f\in L^1\ti{loc}(\R^d).
\end{equation*}
Let $j\in\lbrace1,\ldots,3^{d}\rbrace$. Since $\cS^{j}$ is $\frac{6^{d}+\eta-1}{6^{d}}$-sparse, we deduce
\begin{equation*}
\sum_{\substack{Q\in\cS^{j}\\Q\subseteq R}}|Q|\leq\frac{6^{d}}{1-\eta}|R|,\qquad\forall R\in\cS.
\end{equation*}
Choose an integer $m$ greater than $\max\left(2,\frac{6^d}{1-\eta}-1\right)$. Then, by \cite[Lemma 6.6]{intuitive} we have that one can write $\cS^{j}=\bigcup_{k=1}^{m}\cS^{j,k}$ in such a way that
\begin{equation*}
\sum_{\substack{Q\in\cS^{j,k}\\Q\subseteq R}}|Q|\leq\lambda|R|,\qquad\forall R\in\cS^{j,k},\qquad\forall k=1,\ldots,m,
\end{equation*}
where $\lambda:=1+\frac{\frac{6^d}{1-\eta}-1}{m}$. Since $\lambda\in(1,2)$, setting $\e:=\lambda-1\in(0,1)$ we immediately deduce
\begin{equation*}
\sum_{Q\in\text{ch}\ci{\cS^{j,k}}(R)}|Q|\leq\e|R|,\qquad\forall R\in\cS^{j,k},
\end{equation*}
and thus $\cS^{j,k}$ is martingale $\e$-sparse.

Noting that
\begin{equation*}
\mathcal{A}\ci{\cS^{j},p}f\leq\sum_{k=1}^{m}\mathcal{A}\ci{\cS^{j,k},p}f,\qquad\forall f\in L^1\ti{loc}(\R^d),\qquad\forall j=1,\ldots, 3^d
\end{equation*}
completes the reduction.

\subsection{Estimating a fundamental function} Let $r\in(1,\infty)$. Consider the strictly increasing Young function $\Phi:[0,\infty)\rightarrow[0,\infty)$ given by
\begin{equation*}
\Phi(t):=t\log(e+t)(\log(\log(e^e+t)))^r,\qquad0\leq t<\infty.
\end{equation*}
Let $\Phi^{-1}:[0,\infty)\rightarrow[0,\infty)$ be the inverse of $\Phi$. Consider the function $\phi:[0,1]\rightarrow[0,\infty)$ given by
\begin{equation*}
\phi(s):=\frac{1}{\Phi^{-1}(1/s)},\qquad\forall s\in(0,1],\qquad \psi(0)=0.
\end{equation*}
We prove that
\begin{equation*}
\phi(s)\sim_{r} s(12-\log s)(\log(12-\log s))^r,\qquad\forall s\in[0,1].
\end{equation*}
It suffices to prove that there exists $M_1=M_1(r)>100$ such that
\begin{equation*}
\Phi^{-1}(t)\sim_{r}\frac{t}{(\log t)(\log(\log t))^r},\qquad\forall t\in[M_1,\infty).
\end{equation*}
Set
\begin{equation*}
B(t):=\frac{t}{(\log t)(\log(\log t))^r},\qquad t\in[100,\infty).
\end{equation*}
Choose $M=M(r)>100$ such that $B(t)>100$, for all $t\in(M,\infty)$. Notice that
\begin{equation*}
\Phi(s)\sim_{r} s(\log s)(\log(\log s))^r,\qquad\forall t\in[100,\infty).
\end{equation*}
Set
\begin{equation*}
C(t):=\log(\log t)+r\cdot\log(\log(\log t)),\qquad t\in[100,\infty).
\end{equation*}
Then, we have
\begin{align*}
\Phi(B(t))\sim_{r}\frac{t(\log t-C(t))(\log(\log t-C(t)))^r}{(\log t)(\log(\log t))^r},\qquad\forall t>M,
\end{align*}
therefore since $\lim_{t\rightarrow\infty}\frac{C(t)}{\log t}=0$ we deduce that there exist constants $M_1=M_1(r)>M$ and $c=c(r)<1<C=C(r)$ such that
\begin{equation*}
ct\leq\Phi(B(t))\leq Ct,\qquad\forall t>M_1.
\end{equation*}
Then, since $\Phi$ is convex with $\Phi(0)=0$ we deduce
\begin{equation*}
\Phi\left(\frac{1}{C}B(t)\right)\leq t\leq\Phi\left(\frac{1}{c}B(t)\right),~\forall t>M_1,
\end{equation*}
therefore
\begin{equation*}
\frac{1}{C}B(t)\leq\Phi^{-1}(t)\leq\frac{1}{c}B(t),~\forall t>M_1.
\end{equation*}
This yields the desired result.

\subsection{Elementary estimates for sums of series} Let $r\in(1,2)$. We prove that
\begin{equation*}
\sum_{n=2}^{\infty}(\log n)^{r}x^{n}\lesssim\frac{(-\log(1-x))^{r}}{1-x},\qquad \sum_{n=2}^{\infty}n(\log n)^{r}x^{n}\lesssim\frac{(-\log(1-x))^{r}}{(1-x)^2},\qquad 0<x<1.
\end{equation*}
Recall that
\begin{equation*}
-\log(1-x)=\sum_{n=1}^{\infty}\frac{1}{n}x^n,\qquad \frac{1}{1-x}=\sum_{n=0}^{\infty}x^{n},\qquad\frac{1}{(1-x)^2}=\sum_{n=0}^{\infty}(n+1)x^{n},\qquad0<x<1.
\end{equation*}
Set $a:=r-1\in(0,1)$. Note the following consequences of the integral test:
\begin{equation*}
\sum_{k=1}^{n}\frac{1}{k}\sim\log n,\qquad \sum_{k=1}^{n}\frac{(\log k)^{a}}{k}\sim (\log n)^{r},\qquad n=2,3,\ldots.
\end{equation*}
Series multiplication yields then that for all $0<x<1$ there holds
\begin{equation*}
(-\log (1-x))^2=\sum_{n=2}^{\infty}\left(\sum_{k=1}^{n-1}\frac{1}{k(n-k)}\right)x^{n}=\sum_{n=2}^{\infty}\frac{1}{n}\left(\sum_{k=1}^{n-1}\left(\frac{1}{k}+\frac{1}{n-k}\right)\right)x^n\sim\sum_{n=2}^{\infty}\frac{\log n}{n}x^n.
\end{equation*}
Set $b:=1/a\in(1,\infty)$ and $b':=b/(b-1)=1/(1-a)$. Then, for all $x\in(0,1)$, applying H\"{o}lder's inequality (for series) for the exponents $b,b'$ we deduce
\begin{align*}
\sum_{n=2}^{\infty}\frac{(\log n)^{a}}{n}x^n&=\sum_{n=2}^{\infty}\frac{(\log n)^{a}}{n^{a}}x^{an}\cdot\frac{1}{n^{1-a}}x^{n(1-a)}\leq\left(\sum_{n=2}^{\infty}\frac{\log n}{n}x^n\right)^{1/b}\left(\sum_{n=2}^{\infty}\frac{1}{n}x^{n}\right)^{1/b'}\\
&\lesssim (-\log(1-x))^{2a}(-\log(1-x))^{1-a}=(-\log(1-x))^{r}.
\end{align*}
Therefore, for all $x\in(0,1)$, series multiplication yields
\begin{align*}
\frac{(-\log(1-x))^{r}}{1-x}\gtrsim\left(\sum_{n=2}^{\infty}\frac{(\log n)^a}{n}x^n\right)\left(\sum_{n=0}^{\infty}x^n\right)=\sum_{n=2}^{\infty}\left(\sum_{k=2}^{n}\frac{(\log k)^{a}}{k}\right)x^n\gtrsim\sum_{n=2}^{\infty}(\log n)^{r}x^n.
\end{align*}
Finally, for all $x\in(0,1)$, series multiplication yields
\begin{align*}
\frac{(-\log(1-x))^{r}}{(1-x)^2}\gtrsim\left(\sum_{n=2}^{\infty}\frac{(\log n)^{a}}{n}x^n\right)\left(\sum_{n=0}^{\infty}(n+1)x^n\right)=\sum_{n=2}^{\infty}\left(\sum_{k=2}^{n}\frac{(\log k)^a(n-k+1)}{k}\right)x^n.
\end{align*}
For all $n=2,3,\ldots$, we have
\begin{equation*}
(n+1)\sum_{k=2}^{n}\frac{(\log k)^{a}}{k}\gtrsim (n+1)(\log n)^{r},\qquad \sum_{k=2}^{n}(\log k)^{a}\leq n(\log n)^{a}.
\end{equation*}
Noting that $\lim_{n\rightarrow\infty}\frac{(\log n)^{r}}{(\log n)^{a}}=\lim_{n\rightarrow\infty}\log n=\infty$, we deduce
\begin{equation*}
\frac{(-\log(1-x))^{r}}{(1-x)^2}\gtrsim\sum_{n=2}^{\infty}n(\log n)^{r}x^n,
\end{equation*}
for all $x\in(0,1)$.


\begin{thebibliography}{99}

\bibitem{log-bumps} Theresa C. Anderson, David V. Cruz-Uribe, and Kabe Moen, \emph{Logarithmic Bump Conditions for \cz Operators on Spaces of Homogeneous type}, \textit{Publ. Mat.}, Volume 59, Number 1 (2015), 17-43.

\bibitem{interpolation} Colin Bennett, Robert Sharpley, \emph{Interpolation of operators}, Pure and Applied Mathematics, vol. 129, Academic Press Inc., Boston, MA, 1988

\bibitem{extrapolation} David V. Cruz-Uribe, Jos\'e Maria Martell, and Carlos P\'erez, \emph{Weights, extrapolation and the theory of Rubio de Francia}, Operator Theory: Advances and Applications, vol. 215, Birkh\"{a}user/Springer Basel AG, Basel, 2011.

\bibitem{cruz-uribe-perez} David V. Cruz-Uribe, Carlos P\'{e}rez, \emph{Two-weight, weak-type norm inequalities for fractional integrals, \cz operators and commutators}, \textit{Indiana Univ. Math. J.}, 49 (2000), no. 2, 697--721

\bibitem{cruz-uribe--reznikov-volberg} David V. Cruz-Uribe, Alexander Reznikov, and Alexander Volberg, Logarithmic bump conditions and the two-weight boundedness of \cz operators, \textit{Adv. Math.} 255 (2014), 706--729, DOI 10.1016/j.aim.2014.01.016. MR3167497.

\bibitem{culiuc} Amalia V. Culiuc, A note on two weight bounds for the generalized Hardy--Littlewood Maximal operator, arXiv:1506.07125v1

\bibitem{lacey} Michael T. Lacey, \emph{On the separated bumps conjecture for \cz operators}, \textit{Hokkaido Math. J.}, Volume 45, Number 2 (2016), 223-242.

\bibitem{lacey-A_2} Michael T. Lacey, \emph{An elementary proof of the $A_2$ bound}, \textit{Israel J. of Math.}, March 2017, 217(1), 181--195

\bibitem{lerner} Andrei K. Lerner, \emph{On an estimate of \cz operators by dyadic positive operators}, \textit{J. d' Anal. Mat.}, Oct. 2013, Vol. 121, Issue 1, p. 141--161.

\bibitem{lerner-sparse-pointwise} Andrei K. Lerner, \emph{On pointwise estimates involving sparse operators}, \textit{New York J. of Math.}, 22 (2016), 341--349

\bibitem{intuitive} Andrei K. Lerner and Fedor Nazarov, \emph{Intuitive dyadic calculus: the basics}, \textit{Expositiones Mathematicae} (2018), ISSN 0723-0869, DOI 10.1016/j.exmath.2018.01.001.

\bibitem{convex_body} Fedor Nazarov, Stephanie Petermichl, Sergei Treil, and Alexander Volberg, \emph{Convex Body Domination and Weighted Estimates with Matrix Weights}, \textit{Adv. Math.} 318 (2017), 279--306

\bibitem{nazarov--reznikov--treil--volberg} Fedor Nazarov, Alexander Reznikov, Sergei Treil, and Alexander Volberg, \emph{A Bellman function proof of $L^2$ bump conjecture}, \textit{J. d' Anal. Mat.}, Oct. 2013, Vol. 121, Issue 1, p. 255--277.

\bibitem{neugebauer} Christoph J. Neugebauer, \emph{Inserting $A_p$-weights}, \textit{Proc. Amer. Math. Soc.}, 87(4):644--648, 1983

\bibitem{rahm-spencer} Robert Rahm, Scott Spencer, \emph{Entropy bumps and another sufficient condition for the two-weight boundedness of sparse operators}, \textit{Israel J. of Math.}, Febr. 2018, Vol. 223(1), 197--204

\bibitem{reguera} Maria C. Reguera, \emph{On Muckenhoupt--Wheeden Conjecture}, \textit{Adv. Math.} 227 (2011), 1436--1450

\bibitem{reguera-scurry} Maria C. Reguera and James Scurry, \emph{On Joint Estimates for Maximal Functions and Singular Integrals in Weighted Spaces}, \textit{Proc. Amer. Math. Soc.} 141 (2013), 1705--1717

\bibitem{reguera-thiele} Maria C. Reguera and Christoph Thiele, \emph{The Hilbert Transform Does Not Map $L^1(Mw)$ to $L^{1,\infty}(w)$}, \textit{Math. Res. Let.} (2012), Vol. 19, Number 1

\bibitem{treil} Sergei Treil, \emph{A Remark on Two Weight Estimates for Positive Dyadic Operators}, in: Gr\"{o}chenig K., Lyubarskii Y., Seip K. (eds), \textit{Operator-Related Function Theory and Time-Frequency Analysis} (2015), Abel Symposia, vol 9. Springer, Cham.

\bibitem{entropy} Sergei Treil and Alexander Volberg, \emph{Entropy Conditions in Two Weight Inequalities for Singular Integral Operators}, \textit{Adv. Math.}, Oct. 2016, Volume 301, p. 499--548

\end{thebibliography}
\end{document}